\documentclass[12pt]{amsart}
\usepackage{txfonts}      %{article} was 12pt latex e
\usepackage{amssymb}
\usepackage{eucal}
\usepackage{amsmath}
\usepackage{amscd}
\usepackage{xcolor}
\usepackage{multicol}
\usepackage[all]{xy}           %xypic macro for latex2.09
\usepackage{graphicx}
\usepackage{color}
\usepackage{colordvi}
\usepackage{xspace}
\usepackage{tikz}
\usepackage{makecell}
\usepackage{appendix}
\usepackage{amsthm}

\usepackage{ifpdf}
\ifpdf
\usepackage[colorlinks,final,backref=page,hyperindex]{hyperref}
\else
\usepackage[colorlinks,final,backref=page,hyperindex,hypertex]{hyperref}
\fi

\usepackage[active]{srcltx} %SRC Specials for DVI Searching

\begin{document}

%%%%%%%%%%%%%%%%%%%%%%%% Statements

\newtheorem{thm}{Theorem}[section]
\newtheorem{lem}[thm]{Lemma}
\newtheorem{cor}[thm]{Corollary}
\newtheorem{pro}[thm]{Proposition}
\theoremstyle{definition}
\newtheorem{defi}[thm]{Definition}
\newtheorem{ex}[thm]{Example}
\newtheorem{rmk}[thm]{Remark}
\newtheorem{pdef}[thm]{Proposition-Definition}
\newtheorem{condition}[thm]{Condition}

\renewcommand{\labelenumi}{{\rm(\alph{enumi})}}
\renewcommand{\theenumi}{\alph{enumi}}

\newcommand {\emptycomment}[1]{} %to remove paragraphs

\newcommand{\nc}{\newcommand}
\newcommand{\delete}[1]{}
\newcommand{\wuhao}{\fontsize{10pt}{\baselineskip}\selectfont}
\nc{\todo}[1]{\tred{To do:} #1}

\nc{\tred}[1]{\textcolor{red}{#1}}
\nc{\tblue}[1]{\textcolor{blue}{#1}}
\nc{\tgreen}[1]{\textcolor{green}{#1}}
\nc{\tpurple}[1]{\textcolor{purple}{#1}}
\nc{\tgray}[1]{\textcolor{gray}{#1}}
\nc{\torg}[1]{\textcolor{orange}{#1}}
\nc{\tmag}[1]{\textcolor{magenta}}
\nc{\btred}[1]{\textcolor{red}{\bf #1}}
\nc{\btblue}[1]{\textcolor{blue}{\bf #1}}
\nc{\btgreen}[1]{\textcolor{green}{\bf #1}}
\nc{\btpurple}[1]{\textcolor{purple}{\bf #1}}

%\delete{
    \nc{\mlabel}[1]{\label{#1}}  % Use this to suppress names
    \nc{\mcite}[1]{\cite{#1}}  % Use this to suppress names
    \nc{\mref}[1]{\ref{#1}}  % Use this to suppress names
    \nc{\meqref}[1]{\eqref{#1}}  % Use this to suppress names
    \nc{\mbibitem}[1]{\bibitem{#1}} % Use this to show number
%}

\delete{
    \nc{\mlabel}[1]{\label{#1}          { {\small\tgreen{\tt{{\ }(#1)}}}}}
    % Use this lines to show names
    \nc{\mcite}[1]{\cite{#1}{\small{\tt{{\ }(#1)}}}}  % Use this lines to show names
    \nc{\mref}[1]{\ref{#1}{\small{\tred{\tt{{\ }(#1)}}}}}  % Use this lines to show names
    \nc{\meqref}[1]{\eqref{#1}{{\tt{{\ }(#1)}}}}  % Use this lines to show names
    \nc{\mbibitem}[1]{\bibitem[\bf #1]{#1}} % Use this to show name
}

%    \nc{\mlabel}[1]{  % Use the next two lines to show names
%           { {\small\tgreen{\tt{{\ }(#1)}}}}}

\nc{\cm}[1]{\textcolor{red}{Chengming:#1}}
\nc{\yy}[1]{\textcolor{blue}{Yanyong: #1}}
%\nc{\lit}[2]{\textcolor{blue}{#1}{ \textcolor{purple}{(#2)}}}
%\nc{\lit}[2]{\textcolor{blue}{#1}{}} %use this line instead of the previous one to show only the new changes
\nc{\li}[1]{\textcolor{purple}{#1}}
\nc{\lir}[1]{\textcolor{purple}{Li:#1}}

%%%%%%%% new symbols

\nc{\tforall}{\ \ \text{for all }}
\nc{\hatot}{\,\widehat{\otimes} \,}
\nc{\complete}{completed\xspace}
\nc{\wdhat}[1]{\widehat{#1}}

\nc{\ts}{\mathfrak{p}}
\nc{\mts}{c_{(i)}\ot d_{(j)}}

\nc{\NA}{{\bf NA}}
\nc{\LA}{{\bf Lie}}
\nc{\CLA}{{\bf CLA}}

\nc{\cybe}{CYBE\xspace}
\nc{\nybe}{NYBE\xspace}
\nc{\ccybe}{CCYBE\xspace}

\nc{\ndend}{pre-Novikov\xspace}
\nc{\calb}{\mathcal{B}}
\nc{\rk}{\mathrm{r}}
\newcommand{\g}{\mathfrak g}
\newcommand{\h}{\mathfrak h}
\newcommand{\pf}{\noindent{$Proof$.}\ }
\newcommand{\frkg}{\mathfrak g}
\newcommand{\frkh}{\mathfrak h}
\newcommand{\Id}{\rm{Id}}
\newcommand{\gl}{\mathfrak {gl}}
\newcommand{\ad}{\mathrm{ad}}
\newcommand{\add}{\frka\frkd}
\newcommand{\frka}{\mathfrak a}
\newcommand{\frkb}{\mathfrak b}
\newcommand{\frkc}{\mathfrak c}
\newcommand{\frkd}{\mathfrak d}
\newcommand {\comment}[1]{{\marginpar{*}\scriptsize\textbf{Comments:} #1}}

\nc{\vspa}{\vspace{-.1cm}}
\nc{\vspb}{\vspace{-.2cm}}
\nc{\vspc}{\vspace{-.3cm}}
\nc{\vspd}{\vspace{-.4cm}}
\nc{\vspe}{\vspace{-.5cm}}

%%%%%%%%%%%%%%%%%%%%%%% old symbols

\nc{\disp}[1]{\displaystyle{#1}}
\nc{\bin}[2]{ (_{\stackrel{\scs{#1}}{\scs{#2}}})}  %binomial coeff
\nc{\binc}[2]{ \left (\!\! \begin{array}{c} \scs{#1}\\
    \scs{#2} \end{array}\!\! \right )}  %binomial coeff
\nc{\bincc}[2]{  \left ( {\scs{#1} \atop
    \vspace{-.5cm}\scs{#2}} \right )}  %binomial coeff
\nc{\ot}{\otimes}
\nc{\sot}{{\scriptstyle{\ot}}}
\nc{\otm}{\overline{\ot}}
\nc{\ola}[1]{\stackrel{#1}{\la}}%${\Bbb Z}$

\nc{\scs}[1]{\scriptstyle{#1}} \nc{\mrm}[1]{{\rm #1}}

\nc{\dirlim}{\displaystyle{\lim_{\longrightarrow}}\,}
\nc{\invlim}{\displaystyle{\lim_{\longleftarrow}}\,}

\nc{\bfk}{{\bf k}} \nc{\bfone}{{\bf 1}}
\nc{\rpr}{\circ}
%\nc{\apr}{\cdot}
\nc{\dpr}{{\tiny\diamond}}
\nc{\rprpm}{{\rpr}}

%%%%%%%%%%%%%%%%%%%%% roman fonts, in alphabetic order
\nc{\mmbox}[1]{\mbox{\ #1\ }} \nc{\ann}{\mrm{ann}}
\nc{\Aut}{\mrm{Aut}} \nc{\can}{\mrm{can}}
\nc{\twoalg}{{two-sided algebra}\xspace}
\nc{\colim}{\mrm{colim}}
\nc{\Cont}{\mrm{Cont}} \nc{\rchar}{\mrm{char}}
\nc{\cok}{\mrm{coker}} \nc{\dtf}{{R-{\rm tf}}} \nc{\dtor}{{R-{\rm
tor}}}
\renewcommand{\det}{\mrm{det}}
\nc{\depth}{{\mrm d}}
\nc{\End}{\mrm{End}} \nc{\Ext}{\mrm{Ext}}
\nc{\Fil}{\mrm{Fil}} \nc{\Frob}{\mrm{Frob}} \nc{\Gal}{\mrm{Gal}}
\nc{\GL}{\mrm{GL}} \nc{\Hom}{\mrm{Hom}} \nc{\hsr}{\mrm{H}}
\nc{\hpol}{\mrm{HP}}  \nc{\id}{\mrm{id}} \nc{\im}{\mrm{im}}

\nc{\incl}{\mrm{incl}} \nc{\length}{\mrm{length}}
\nc{\LR}{\mrm{LR}} \nc{\mchar}{\rm char} \nc{\NC}{\mrm{NC}}
\nc{\mpart}{\mrm{part}} \nc{\pl}{\mrm{PL}}
\nc{\ql}{{\QQ_\ell}} \nc{\qp}{{\QQ_p}}
\nc{\rank}{\mrm{rank}} \nc{\rba}{\rm{RBA }} \nc{\rbas}{\rm{RBAs }}
\nc{\rbpl}{\mrm{RBPL}}
\nc{\rbw}{\rm{RBW }} \nc{\rbws}{\rm{RBWs }} \nc{\rcot}{\mrm{cot}}
\nc{\rest}{\rm{controlled}\xspace}
\nc{\rdef}{\mrm{def}} \nc{\rdiv}{{\rm div}} \nc{\rtf}{{\rm tf}}
\nc{\rtor}{{\rm tor}} \nc{\res}{\mrm{res}} \nc{\SL}{\mrm{SL}}
\nc{\Spec}{\mrm{Spec}} \nc{\tor}{\mrm{tor}} \nc{\Tr}{\mrm{Tr}}
\nc{\mtr}{\mrm{sk}}

%%%%%%%%%%%%%%%%%% bold face
\nc{\ab}{\mathbf{Ab}} \nc{\Alg}{\mathbf{Alg}}

%%%%%%%%%%%%%%%%%%%Bbb fonts
\nc{\BA}{{\mathbb A}} \nc{\CC}{{\mathbb C}} \nc{\DD}{{\mathbb D}}
\nc{\EE}{{\mathbb E}} \nc{\FF}{{\mathbb F}} \nc{\GG}{{\mathbb G}}
\nc{\HH}{{\mathbb H}} \nc{\LL}{{\mathbb L}} \nc{\NN}{{\mathbb N}}
\nc{\QQ}{{\mathbb Q}} \nc{\RR}{{\mathbb R}} \nc{\BS}{{\mathbb{S}}} \nc{\TT}{{\mathbb T}}
\nc{\VV}{{\mathbb V}} \nc{\ZZ}{{\mathbb Z}}

%%%%%%%%%%%%%%%%%%% cal fonts

\nc{\calao}{{\mathcal A}} \nc{\cala}{{\mathcal A}}
\nc{\calc}{{\mathcal C}} \nc{\cald}{{\mathcal D}}
\nc{\cale}{{\mathcal E}} \nc{\calf}{{\mathcal F}}
\nc{\calfr}{{{\mathcal F}^{\,r}}} \nc{\calfo}{{\mathcal F}^0}
\nc{\calfro}{{\mathcal F}^{\,r,0}} \nc{\oF}{\overline{F}}
\nc{\calg}{{\mathcal G}} \nc{\calh}{{\mathcal H}}
\nc{\cali}{{\mathcal I}} \nc{\calj}{{\mathcal J}}
\nc{\call}{{\mathcal L}} \nc{\calm}{{\mathcal M}}
\nc{\caln}{{\mathcal N}} \nc{\calo}{{\mathcal O}}
\nc{\calp}{{\mathcal P}} \nc{\calq}{{\mathcal Q}} \nc{\calr}{{\mathcal R}}
\nc{\calt}{{\mathcal T}} \nc{\caltr}{{\mathcal T}^{\,r}}
\nc{\calu}{{\mathcal U}} \nc{\calv}{{\mathcal V}}
\nc{\calw}{{\mathcal W}} \nc{\calx}{{\mathcal X}}
\nc{\CA}{\mathcal{A}}

%%%%%%%%%%%%%%%%%%  frak fonts
\nc{\fraka}{{\mathfrak a}} \nc{\frakB}{{\mathfrak B}}
\nc{\frakb}{{\mathfrak b}} \nc{\frakd}{{\mathfrak d}}
\nc{\oD}{\overline{D}}
\nc{\frakF}{{\mathfrak F}} \nc{\frakg}{{\mathfrak g}}
\nc{\frakm}{{\mathfrak m}} \nc{\frakM}{{\mathfrak M}}
\nc{\frakMo}{{\mathfrak M}^0} \nc{\frakp}{{\mathfrak p}}
\nc{\frakS}{{\mathfrak S}} \nc{\frakSo}{{\mathfrak S}^0}
\nc{\fraks}{{\mathfrak s}} \nc{\os}{\overline{\fraks}}
\nc{\frakT}{{\mathfrak T}}
\nc{\oT}{\overline{T}}
%\nc{\frakx}{{\mathfrak x}}
\nc{\frakX}{{\mathfrak X}} \nc{\frakXo}{{\mathfrak X}^0}
\nc{\frakx}{{\mathbf x}}
%\nc{\frakTxo}{{\frakTx}^0}
\nc{\frakTx}{\frakT}      %All rooted trees, correspond to \ncsha(X)
\nc{\frakTa}{\frakT^a}        % rooted trees for \ncsha(A)
\nc{\frakTxo}{\frakTx^0}   % rooted trees for \ncshao(X)
\nc{\caltao}{\calt^{a,0}}   % rooted trees for \ncshao(A)
\nc{\ox}{\overline{\frakx}} \nc{\fraky}{{\mathfrak y}}
\nc{\frakz}{{\mathfrak z}} \nc{\oX}{\overline{X}}

\font\cyr=wncyr10

%%%%%%%%%%%%%%%%%%%%%%%%%%%%%%%%%%%%%%%%%%%%%%%%%%%%%%%%%%%%%%%%%%

\title[Deformation families of Novikov bialgebras via differential bialgebras]{
Deformation families of Novikov bialgebras via differential antisymmetric infinitesimal bialgebras}

\author{Yanyong Hong}
\address{School of Mathematics, Hangzhou Normal University,
Hangzhou 311121, PR China}
\email{yyhong@hznu.edu.cn}

\author{Chengming Bai}
\address{Chern Institute of Mathematics \& LPMC, Nankai University, Tianjin 300071, PR China}
\email{baicm@nankai.edu.cn}

\author{Li Guo}
\address{Department of Mathematics and Computer Science, Rutgers University, Newark, NJ 07102, United States}
         \email{liguo@rutgers.edu}

\subjclass[2010]{16T10, 16T25, 16S80, 17B38, 13N10, 16W99, 17A30, 17D25}
\keywords{Novikov algebra, deformation, Novikov bialgebra, Yang-Baxter
equation, differential algebra, antisymmetric infinitesimal bialgebra,  Zinbiel algebra, $\mathcal{O}$-operator}

\begin{abstract}
Generalizing S. Gelfand's classical construction of a Novikov algebra from a commutative differential algebra, a deformation family $(A,\circ_q)$, for scalars $q$, of Novikov algebras is constructed from what we call an admissible commutative differential algebra, by adding a second linear operator to the commutative differential algebra with certain admissibility condition. The case of $(A,\circ_0)$ recovers the construction of S. Gelfand.
This admissibility condition also ensures a bialgebra theory of commutative differential algebras, enriching the antisymmetric infinitesimal bialgebra. This way, a deformation family of Novikov bialgebras is obtained, under the further condition that the two operators are bialgebra derivations. As a special case, we obtain a bialgebra variation of S. Gelfand's construction with an interesting twist: every commutative and cocommutative differential antisymmetric infinitesimal bialgebra gives rise to a Novikov bialgebra whose underlying Novikov algebra is $(A,\circ_{-\frac{1}{2}})$ instead of $(A,\circ_0)$.
The close relations of the classical bialgebra theories with Manin triples, classical Yang-Baxter type equations,
$\mathcal{O}$-operators, and pre-structures are expanded to the
two new bialgebra theories, in a way that is compatible with the
just established connection between the two bialgebras. As an
application, Novikov bialgebras are obtained from admissible
differential Zinbiel algebras.
\end{abstract}

\maketitle

\vspace{-1.2cm}

\tableofcontents

\vspace{-1.4cm}

\allowdisplaybreaks

\section{Introduction}
Novikov algebras were introduced in connection with Hamiltonian
operators in the formal variational calculus~\mcite{GD1, GD2} and
Poisson brackets of hydrodynamic type~\mcite{BN}.  Moreover, there is also a correspondence
between Novikov algebras and a class of Lie conformal algebras
\mcite{X1}, which was introduced by V. Kac to give an axiomatic
description of the singular part of operator product expansion (or
rather its Fourier transform) of chiral fields in conformal field
theory \mcite{K1}. Novikov algebras are also an important subclass of pre-Lie algebras,
which are closely related to many fields in mathematics and physics such as  affine manifolds and affine
structures on Lie groups \mcite{Ko},  convex homogeneous cones \mcite{V}, deformation of
associative algebras \mcite{Ger} and vertex algebras \mcite{BK, BLP}.

A classical construction of a Novikov algebra is the one given by S. Gelfand~\mcite{GD1} from a commutative associative algebra $(A, \cdot)$ with a derivation $D$ (i.e., a commutative differential algebra): define
\vspb
\begin{equation}
a\circ b:= a\cdot D(b), \ a, b\in A.
\mlabel{eq:diffprod}
\end{equation}
Then $(A,\circ)$ is a Novikov algebra.

The notion of admissible commutative differential algebras was introduced in a recent study of a bialgebra theory for commutative differential algebras~\mcite{LLB}. It is defined to be a commutative differential algebra $(A,\cdot, D)$ equipped with another linear operator $Q$ satisfying the admissibility condition $Q(a\cdot b)=Q(a)\cdot b-a\cdot D(b)$ for all $a, b\in A$ (see Definition~\mref{de:addiff}). In this paper, we apply this admissibility condition to give a deformation family of Novikov algebras, and a deformation family of Novikov bialgebras under an additional condition.

First we show that an admissible commutative differential algebra gives a parametric family of Novikov algebras, which can be regarded as an infinitesimal deformation of the above construction of S. Gelfand (see Proposition \mref{constr2}):
\vspace{-.1cm}
{\wuhao
\begin{equation}
    \begin{split}
        \xymatrix{
            \text{admissible commutative differential algebras} \ar[r]     & \text{deformation families of Novikov algebras}} 
        \end{split}
    \mlabel{eq:algdiag}
\end{equation}
}

Our next goal is to lift this construction to the level of
bialgebras. For this purpose, we first note that the desired bialgebra theory
of admissible commutative differential algebras, built on the
existing antisymmetric infinitesimal bialgebras (ASI
bialgebras)~\mcite{A1,A2,Bai}, is already provided by the
bialgebra theory of commutative differential algebras developed
in~\mcite{LLB}. This theory works because, in addition to ensuring
the construction of a deformation family of Novikov algebras, the
above admissibility condition between $D$ and $Q$ on commutative
differential algebras, when paired with its coalgebra counterpart,
also fulfills the requirement of the bialgebra theory of
commutative differential algebras, so that the underlying
(co)commutative differential (co)algebras are admissible (see
Definition~\mref{de:admdifbialg}).

With the bialgebra structure of admissible commutative
differential algebras in place, we move on to establish its
connection with the Novikov bialgebras as developed in our recent
work \mcite{HBG} in the construction of infinite-dimensional Lie
bialgebras. Under the additional conditions that the pair of
linear operators in the admissible commutative differential
algebra are both derivations and coderivations (that is, derivations on the ASI bialgebras \cite{LLB}),  a
commutative and cocommutative differential ASI bialgebra gives
rise to a deformation family of Novikov bialgebras, enriching
Diagram \meqref{eq:algdiag} to the level of bialgebras (see
Theorem \mref{bialgebra-constr}):
\vspace{-.1cm}
{\wuhao
\begin{equation}
\begin{split}
\xymatrix{
\txt{commutative and cocommutative \\
    differential ASI bialgberas} \ar[r]     & \txt{deformation families \\ of Novikov bialgebras}}
\end{split}
    \mlabel{eq:algdiag2}
\end{equation}
}
In the special case when the deformation parameter $q$ is $-\frac{1}{2}$, a commutative and cocommutative differential ASI bialgebra {\it unconditionally} gives rise to a Novikov bialgebra. Here we observe an unexpected departure from the classical construction of S. Gelfand, in the sense that the resulting Novikov bialgebra has its underlying Novikov algebra to be the deformed Novikov algebra $(A,\circ_{-\frac{1}{2}})$, instead of the Novikov algebra $(A,\circ_0)$ from the commutative differential algebra originally constructed by S. Gelfand.

As shown in~\mcite{HBG}, the importance of Novikov bialgebras are also reflected by their characterization by Manin triples of Novikov
algebras. Furthermore, antisymmetric
solutions of Novikov Yang-Baxter equation (NYBE) naturally give
rise to Novikov bialgebras. In turn, such solutions are provided
by $\mathcal{O}$-operators on Novikov algebras. Moreover,
pre-Novikov algebras naturally provide $\mathcal{O}$-operators on the descendent Novikov algebras. These connections are depicted in the  diagram \vspace{-.1cm}
\begin{equation}
    \begin{split}
        \xymatrix{
            \text{pre-Novikov}\atop\text{ algebras} \ar[r]     & \mathcal{O}\text{-operators on}\atop\text{Novikov algebras}\ar[r] &
            \text{solutions of}\atop \text{NYBE} \ar[r]  & \text{Novikov}\atop \text{ bialgebras}  \ar@{<->}[r] &
            \text{Manin triples of} \atop \text{Novikov algebras} }
    \end{split}
    \mlabel{eq:Novdiag}
\end{equation}

To further understand the implication in Diagram~\meqref{eq:algdiag2}, for each structure in Diagram~\meqref{eq:Novdiag}, we recall from~\cite{LLB} its counterpart for commutative and cocommutative differential ASI bialgebras, including their characterization by double constructions of differential Frobenius algebras, their corresponding admissible associative
Yang-Baxter equation, their $\mathcal{O}$-operators and the admissible differential Zinbiel algebras.
Then the analogy of the implication in Diagram ~\meqref{eq:algdiag2} is established for each pair of the structures thus obtained in a consistent manner, culminating in the commutative diagram
\vspace{-.1cm}
\small{
\begin{equation}
    \begin{split}\xymatrix{
       \txt{\tiny admissible\\ \tiny diff. Zinbiel \\ \tiny algebras } \ar[r]\ar@{->}[d]     & \txt{\tiny $\mathcal{O}$-operators on\\  \tiny admissible comm.\\ \tiny diff. algebras} \ar[r]\ar@{->}[d] &
        \txt{ \tiny solutions \\ \tiny of admissible\\ \tiny AYBE}  \ar[r]\ar@{->}[d]  & \text{comm. cocomm. }\atop \text{diff. ASI bialgeras}   \ar@{<->}[r]\ar@{->}[d] &
      \text{double constructions of} \atop \text{diff. Frobenius algebras}\ar@{->}[d]\\
            \text{pre-Novikov}\atop\text{ algebras} \ar[r]       & \mathcal{O}\text{-operators of}\atop\text{Novikov algebras}\ar[r]  &
            \text{solutions of}\atop \text{NYBE} \ar[r]   & \text{Novikov}\atop \text{ bialgebras}  \ar@{<->}[r]   &
        \text{Manin triples of} \atop \text{Novikov algebras}  }
\vspace{-.1cm}
 \end{split}
    \mlabel{eq:bigcommdiag}
\end{equation}}
Here each entry in the bottom row is meant to either be a deformation family in $q\in \bfk$ under certain conditions, or be $q=-\frac{1}{2}$ unconditionally. Here the commutativity of the first (resp. second, resp. third, resp. fourth) square follows from Proposition~\mref{pp:preoop} (resp. Proposition~\mref{pro:o-ybe}, resp. Proposition~\mref{pp:rmatbialg}, resp. Proposition~\mref{corr-bialgebra}).

The first benefit from this commutative diagram is that the complicated argument that warrants the extra condition for the implication between the two types of bialgebras in column four becomes transparent in the perspective of the double constructions and Manin triples in column five (see Remark~\mref{rk:square4}).

This diagram also gives a natural construction of Novikov bialgebras from admissible differential Zinbiel algebras (see Theorem~\mref{Constr-Nov-diff-Zinb-2}).
Furthermore, thanks to the construction of completed Lie bialgebras from Novikov bialgebras in~\mcite{HBG}, we also obtain a construction of completed Lie bialgebras from commutative and cocommutative differential ASI bialgebras (see Theorem~\mref{thm-con-Lie-bi}).

We now give an outline of the paper. In Section \mref{s:bialg}, we generalize S. Gelfand's construction of Novikov algebras to deformation families of Novikov algebras from admissible commutative differential algebras. We further present a construction of deformation families of Novikov bialgebras from commutative and cocommutative differential ASI bialgebras. Several examples are also given. As an application, we give a construction of completed Lie bialgebras from commutative and cocommutative differential ASI bialgebras. Moreover, we give an interpretation of the Novikov bialgebra construction from the viewpoint of Manin triples (or double constructions).

In Section \mref{s:ybe}, we give a construction of
antisymmetric solutions of the Novikov Yang-Baxter equation
from antisymmetric  solutions of
the admissible associative Yang-Baxter equation in an admissible
commutative differential algebra. The relationship between
$\mathcal{O}$-operators on admissible commutative differential
algebras and those on the corresponding Novikov algebras is
established. Moreover, we obtain a construction of pre-Novikov algebras
from admissible differential Zinbiel algebras. Finally, we show
that there are natural constructions of Novikov bialgebras from
admissible differential Zinbiel algebras.

\smallskip
\noindent
{\bf Notations.}
Throughout this paper, we fix a base field ${\bf k}$ of characteristic $0$. All vector spaces and algebras a over $\bfk$. Unless otherwise stated, they  are assumed to be finite-dimensional even though many results still hold in the infinite-dimensional cases. Let
$$\tau:A\otimes A\rightarrow A\otimes A,\quad a\otimes b\mapsto b\otimes a,\;\;\;\; a,b\in A,$$
be the flip operator. The identity map is denoted by $\id$. Let $A$ be a vector space with a binary operation $\circ$.
Define linear maps
$L_{\circ}, R_{\circ}:A\rightarrow {\rm End}_{\bf k}(A)$ by
\begin{eqnarray*}
L_{\circ}(a)b=a\circ b,\;\; R_{\circ}(a)b=b\circ a, \;\;\;a, b\in A.
\vspb
\end{eqnarray*}
\section{Deformation families of Novikov bialgebras via admissible commutative differential bialgebras}
\mlabel{s:bialg}

In this section, we introduce a construction of Novikov algebras
from admissible commutative differential algebras, which can be
seen as an infinitesimal deformation of the classical construction
of S. Gelfand. Then we lift this construction to the level of
bialgebras, showing that there is also a natural construction of
Novikov bialgebras from commutative and cocommutative differential
antisymmetric infinitesimal bialgebras. Moreover, the relationship
between the corresponding Manin triples (or double constructions) is
also investigated.
\vspb
\subsection{Novikov algebras via admissible commutative differential algebras}

Recall that a {\bf Novikov algebra} is a vector space $A$ with a binary operation $\circ$ that satisfies
\begin{eqnarray}
\mlabel{lef}(a\circ b)\circ c-a\circ (b\circ c)&=&(b\circ a)\circ c-b\circ (a\circ c),\\
\mlabel{Nov}(a\circ b)\circ c&=&(a\circ c)\circ b, \;\;\;a, b, c\in A.
\end{eqnarray}
Denote it by $(A,\circ)$. Also a {\bf differential algebra} is an associative algebra $(A,\cdot)$ equipped with a {\bf derivation} $D$, defined by the abstract Leibniz rule
\vspb
\begin{equation}
D(a\cdot b)=D(a)\cdot b+a\cdot D(b),\;\;\;a, b\in A.
\end{equation}
Denote it by $(A, \cdot, D)$. A differential algebra $(A, \cdot, D)$ is called {\bf commutative} if $(A, \cdot)$ is commutative.

There is the following classical construction of Novikov algebras.
\begin{pro} {\rm (\mcite{GD1}, S. Gelfand)} \mlabel{lem:cons} Let $(A, \cdot, D)$ be a commutative differential algebra. Define a binary operation $\circ$ on $A$ by
\vspb
\begin{equation}\mlabel{eq:cons}
a\circ b\coloneqq a\cdot D(b),\;\;\;\; a,b\in A.
\end{equation}
Then $(A,\circ)$ is a Novikov algebra, called the {\bf Novikov algebra induced from $(A, \cdot, D)$}.
\end{pro}

We next recall the infinitesimal deformation of Novikov algebras.

\begin{lem}\mcite{BM} \mlabel{inf-defor}
Let $(A, \circ)$ be a Novikov algebra and $\circ_q$ be a binary operation on $A$ defined by
\begin{eqnarray}\mlabel{eq:deform}
a\circ_q b:=a\circ b+qf(a,b), \;\; a, b\in A,
\end{eqnarray}
where $f:A\otimes A\rightarrow A$ is a linear map and $q\in {\bf k}$. Then $(A, \circ_q)$ is a Novikov algebra for all $q\in {\bf k}$ if and only if
\vspb
\begin{eqnarray}
&&\mlabel{eq:defor-1}f(f(a,b),c)-f(a, f(b,c))-f(f(b,a),c)+f(b,f(a,c))=0,\\
&&f(a, b\circ c)-f(a\circ b, c)+f(b\circ a, c)-f(b, a\circ c)+a\circ f(b,c)-f(a,b)\circ c\nonumber\\
&&\mlabel{eq:defor-2}\qquad\qquad\;\;\;\;\;\;\;+f(b,a)\circ c-b\circ f(a, c)=0,\\
&&\mlabel{eq:defor-3}f(f(a,b),c)=f(f(a,c),b),\\
&&\mlabel{eq:defor-4}f(a,b)\circ c-f(a, c)\circ b+f(a\circ b, c)-f(a\circ c, b) =0, \;\;\;a, b, c\in A.
\end{eqnarray}
In this case, $(A,\circ_q)$ is called an {\bf infinitesimal
deformation of $(A,\circ)$}.
\end{lem}

We now give an infinitesimal deformation of S. Gelfand's construction of Novikov algebras.

\begin{pro}\mlabel{spec-def}
Let $(A, \cdot, D)$ be a commutative differential algebra and
$Q:A\rightarrow A$ be a linear map. Define a binary operation
$\circ_q$ on $A$ by
\vspb
\begin{eqnarray}
a\circ_q b:=a\cdot D(b)+q a\cdot Q(b) (= a\cdot (D+qQ)(b)),\;\;\;\;a, b\in A,
\end{eqnarray}
where $q\in {\bf k}$. Then $(A, \circ_q)$ is a Novikov algebra for all $q\in {\bf k}$ if and only if
\begin{eqnarray}
&&\mlabel{eq:defor-5}(a\cdot Q(b))\cdot Q(c)-a\cdot Q(b\cdot Q(c))=(b\cdot Q(a))\cdot Q(c)-b\cdot Q(a\cdot Q(c)),\\
&&\mlabel{eq:defor-6}(a\cdot Q(b))\cdot D(c)-a\cdot Q(b\cdot D(c))=(b\cdot Q(a))\cdot D(c)-b\cdot Q(a\cdot D(c)),\;\;\;\; a, b, c\in A.
\end{eqnarray}
\end{pro}
\begin{proof}
Take $f(a,b)=a\cdot Q(b)$ and $a\circ b=a\cdot D(b)$ in Lemma \mref{inf-defor}. Then Eqs. (\mref{eq:defor-3}) and (\mref{eq:defor-4}) naturally hold, and Eqs. (\mref{eq:defor-1}) and (\mref{eq:defor-2}) hold if and only if Eqs. (\mref{eq:defor-5}) and (\mref{eq:defor-6}) hold respectively. This completes the proof.
\end{proof}

\begin{ex}
Proposition~\mref{spec-def} has the following special cases.
\begin{enumerate}
\item If $Q$ is also a derivation, then the conclusion also
follows from the fact that the linear combination of two derivations is still a
derivation.
\item Take $Q=\id$ or $Q(b)=x\cdot b, b\in A$ for a
fixed $x\in A$. Then Eqs. (\mref{eq:defor-5}) and
(\mref{eq:defor-6}) naturally hold and hence $(A, \circ_q)$ is a
Novikov algebra. The two constructions are exactly the ones given
in \mcite{F} by Filipov and  \mcite{X2} by Xu respectively.
\end{enumerate}
\end{ex}

Next, we give a new construction of Novikov algebras from admissible commutative differential algebras.

\begin{defi}  \mcite{LLB}
Let $(A, \cdot, D)$ be a commutative differential algebra and $Q: A\rightarrow A$ be a linear map. If
\vspb
\begin{eqnarray}
\mlabel{eq:RSI1} Q(a\cdot b)=Q(a)\cdot b-a\cdot D(b),\;\;\;\;\;a, b\in A,
\end{eqnarray}
then we say that $Q$ is  {\bf admissible to $(A, \cdot, D)$} or that $(A, \cdot, D)$ is {\bf $Q$-admissible}. We also say that the quadruple $(A, \cdot, D, Q)$ is an {\bf admissible commutative differential algebra}.
\mlabel{de:addiff}
\end{defi}

\begin{rmk} By \mcite{LLB}, a linear map $Q:A\rightarrow A$ is admissible
to a commutative differential algebra $(A, \cdot, D)$ if and only
if $(A^*,-L_{\cdot}^*,Q^*)$ is a representation of $(A,\cdot,D)$
(see Definition~\ref{Def:mo}).
\end{rmk}

\begin{ex}
Note that $(A, \cdot, D)$ is always $-D$-admissible. Moreover, if
$(A, \cdot, D, Q)$ is an admissible commutative differential
algebra, then $Q(a\cdot b)=a\cdot Q(b)-D(a)\cdot b$ for all
$a,b\in A$, due to the commutativity of $\cdot$. Therefore, we
obtain
\vspb
\begin{eqnarray}
a\cdot (D+Q)(b)=(D+Q)(a)\cdot b, \;\;\;\; a, b\in A.
\end{eqnarray}
\end{ex}

\begin{pro}\mlabel{constr2}
Let $(A, \cdot, D, Q)$ be an admissible commutative differential
algebra. Define a binary operation $\circ_q$ on $A$ by
\vspb
\begin{eqnarray}\mlabel{eq:k}
a\circ_q b:=a\cdot D(b)+q a\cdot Q(b)=a\cdot (D+qQ) b, \;\; a,
b\in A,
\end{eqnarray}
where $q\in {\bf k}$. Then $(A, \circ_q)$ is a Novikov algebra for all $q\in {\bf k}$.
 We call $(A, \circ_q)$ the {\bf Novikov algebra induced from $(A, \cdot, D, Q)$}.
\end{pro}

\begin{proof}
It is straightforward to check that in this case, Eqs.
(\mref{eq:defor-5}) and (\mref{eq:defor-6}) hold. Then $(A,
\circ_q)$ is a Novikov algebra by Proposition \mref{spec-def}.
\end{proof}

\begin{rmk}
Note that $(A, \circ_0)$ is exactly the Novikov algebra obtained
by S. Gelfand's construction and hence  $(A, \circ_q)$ is an
infinitesimal deformation of $(A, \circ_0)$.
\end{rmk}

\begin{rmk}Let $(A, \cdot, D, Q)$ be an admissible commutative differential algebra. More generally, for each $(p,q)\in \bfk^2$, the product $a \circ_{p,q} b:= a\cdot(pD+qQ)(b)$ also defines a Novikov algebra. The case of
$\circ_{1,q}$ recovers the product $\circ_q$ defined in Eq.~\meqref{eq:k};
while the case of $\circ_{-2,1}$ was observed by Guilai Liu and we thank him for communicating it to us.
\end{rmk}

\subsection{Novikov bialgebras and commutative and cocomutative differential ASI bialgebras}
We start with recalling the related notions.
\begin{defi}\mcite{HBG, KUZ}
A {\bf Novikov coalgebra} is a vector space $A$ with a linear map $\Delta: A\rightarrow A\otimes A$ such
that
\vspb
\begin{eqnarray}
\mlabel{Lc3}(\id\otimes \Delta)\Delta(a)-(\tau\otimes \id)(\id\otimes \Delta)\Delta(a)&=&(\Delta\otimes \id)\Delta(a)-(\tau\otimes \id)(\Delta\otimes \id)\Delta(a),\\
\mlabel{Lc4}(\tau\otimes \id)(\id\otimes \Delta)\tau
\Delta(a)&=&(\Delta\otimes \id)\Delta(a),\;\;a\in A.
\end{eqnarray}
\end{defi}

Let $\langle \cdot, \cdot \rangle$ be the usual pairing between a vector space and its dual vector space. Let $V$ and $W$ be vector spaces. For a linear map $\varphi: V\rightarrow W$, let $\varphi^\ast: W^\ast\rightarrow V^\ast$ be the transpose map, characterized by
\begin{eqnarray*}
\langle \varphi^\ast(f),v\rangle=\langle f, \varphi(v)\rangle,\;\;\; f\in W^\ast, v\in V.
\end{eqnarray*}
Note that $(A, \Delta)$ is a Novikov coalgebra if and only if $(A^\ast, \Delta^\ast)$ is a Novikov algebra.

Let $(A, \circ)$ be a Novikov algebra. Throughout the paper, we set
\vspa
$$a\star b:=a\circ b+b\circ a, \quad a, b\in A.
\vspa
$$

\begin{defi}\mcite{HBG}
Let $(A,\circ)$ be a Novikov algebra and $(A, \Delta)$ be a
Novikov coalgebra. The triple $(A,\circ,\Delta)$ is called a {\bf Novikov
bialgebra} if the following compatibility conditions hold.  {\small\begin{eqnarray}
&\mlabel{Lb5}\Delta(a\circ b)=(R_{\circ} (b)\otimes \id)\Delta (a)+(\id\otimes L_{\star}(a))(\Delta(b)+\tau\Delta(b)),&\\
&\mlabel{Lb6}(L_{\star}(a)\otimes \id)\Delta(b)-(\id\otimes L_{\star}(a))\tau\Delta (b)=(L_{\star}(b)\otimes \id)\Delta(a)-(\id\otimes L_{\star}(b))\tau\Delta(a),&\\
&\mlabel{Lb7}(\id\otimes R_{\circ}(a)-R_{\circ}(a)\otimes
\id)(\Delta(b)+\tau\Delta(b))=(\id\otimes
R_{\circ}(b)-R_{\circ}(b)\otimes
\id)(\Delta(a)+\tau\Delta(a)), \ a, b\in A.&
\end{eqnarray}}
\end{defi}

Recall that a {\bf coassociative coalgebra} is a vector space $A$ with a linear map $\delta:A\rightarrow A\otimes A$ satisfying
\begin{eqnarray*}
(\delta\otimes \id)\delta(a)=(\id\otimes \delta)\delta(a),\;\;\; a\in A.
\end{eqnarray*}
A coassociative coalgebra $(A, \delta)$ is called {\bf cocommutative} if $\delta=\tau\delta$.

For a coassociative coalgebra $(A, \delta)$, a linear
map $Q: A\rightarrow A$ is called a {\bf coderivation} on $(A,
\delta)$ if
\begin{eqnarray}
\delta (Q(a))=(\id \otimes Q+Q\otimes
\id)\delta(a),\;\;\;a\in A.
\end{eqnarray}
A {\bf cocommutative differential (coassociative) coalgebra} is a triple $(A,\delta,Q)$ where $(A,\delta)$ is a cocommutative coassociative coalgebra and $Q$ is a coderivation.
Obviously, $(A,\delta,Q)$ is a cocommutative differential
coalgebra if and only if $(A^*,\delta^*,Q^*)$ is a commutative differential algebra.

Let $(A, \cdot)$ be an associative algebra and $(A, \delta)$ be a coassociative coalgebra.
If moreover,
\begin{eqnarray}
&&\mlabel{ASI1}\delta(a\cdot b)=(\id\otimes L_{\cdot}(a))\delta(b)+(R_{\cdot}(b)\otimes \id)\delta(a),\\
&&\mlabel{ASI2}(L_{\cdot}(b)\otimes \id-\id\otimes
R_{\cdot}(b))\delta(a)+\tau((L_{\cdot}(a)\otimes
\id-\id\otimes R_{\cdot}(a))\delta(b))=0,\;\;\;\;a, b\in A,
\end{eqnarray}
then $(A, \cdot, \delta)$ is called an {\bf antisymmetric
infinitesimal bialgebra} or simply an {\bf ASI bialgebra} (see
\mcite{Bai}). Note that when $(A, \cdot)$ is commutative and $(A,
\delta)$ is cocommutative, Eq.~(\mref{ASI2}) automatically holds.

\begin{defi}
Let $(A,\delta,Q)$ be a cocommutative differential
coalgebra. Let $D:A\rightarrow A$ be a linear map satisfying
\vspb
\begin{equation}\mlabel{eq:RSI2}  (D\otimes \id-\id \otimes Q)\delta(a)=\delta (D(a)),\;\;\;\;\; a\in A.
\end{equation}
We call $(A,\delta, Q,D)$ an {\bf admissible cocommutative differential
coalgebra}.
\end{defi}

Then the following result is obvious.

\begin{pro}\mlabel{pro:dual-adm}
$(A,\delta, Q,D)$ is an admissible cocommutative differential
coalgebra if and only if $(A^*$, $\delta^*$, $Q^*,$ $D^*)$ is an admissible commutative differential algebra.
\end{pro}

Next, we present a construction of Novikov coalgebras from admissible cocommutative differential
coalgebras.
\begin{pro}\mlabel{pro:co-cons}
Let $(A,\delta, Q, D)$ be an admissible cocommutative differential
coalgebra. Define a coproduct $\Delta_q$ on $A$ by
\vspb
\begin{equation} \mlabel{eq:cons2}
\Delta_q(a):=(\id\otimes (Q+qD))\delta(a),\;\;\;\;a\in A,
\end{equation}
where $q\in {\bf k}$. Then $(A,\Delta_q)$ is a Novikov coalgebra for all $q\in {\bf k}$. We say that $(A, \Delta_q)$ is  the {\bf Novikov coalgebra induced from $(A,\delta, Q, D)$}.
\end{pro}

\begin{proof}
Let $f$, $g\in A^\ast$ and $a\in A$. Then we have
\vspb
\begin{eqnarray*}
&\langle f\otimes g, \Delta_q(a) \rangle = \langle f\otimes g, (\id\otimes (Q+qD))\delta(a)\rangle
= \langle \delta^\ast(f\otimes (Q^\ast+qD^\ast)g), a\rangle.&
\end{eqnarray*}
Let $f\circ_q g:=\delta^\ast(f\otimes (Q^\ast+qD^\ast)g)$ for all $f$, $g\in A^\ast$. By Proposition \mref{pro:dual-adm}, $(A^*,\delta^*, Q^*,$ $D^*)$ is an admissible commutative differential algebra. Then by Proposition \mref{constr2}, $(A^\ast, \circ_q)$ is a Novikov algebra. Therefore,
$(A, \Delta_q)$ is a Novikov coalgebra.
\end{proof}
\vspb
\begin{rmk}
This result can be also proved by a direct
checking. Therefore, it holds when $A$
is infinite-dimensional.
\vspb
\end{rmk}
\begin{rmk}
The special case that $(A, \Delta_0)$ is a Novikov coalgebra was obtained in \mcite{KUZ}. On the other hand, let $(A, \Delta)$ be a Novikov coalgebra and $\Delta_q$ be a coproduct on $A$ defined by
\begin{eqnarray}\mlabel{eq:deformc}
\Delta_q(a):=\Delta(a)+q \Theta(a), \;\; a\in A,
\end{eqnarray}
where $\Theta: A\rightarrow A\otimes A$ is a linear map and $q\in {\bf k}$. Then $(A,\Delta_q)$ is called an {\bf infinitesimal deformation} of $(A,\Delta)$ if $(A, \Delta_q)$ is a Novikov coalgebra for all $q\in {\bf k}$.
In this sense,  the Novikov coalgebra $(A, \Delta_q)$ defined by Eq.~(\mref{eq:cons2}) is an infinitesimal deformation of $(A, \Delta_0)$.
\end{rmk}

Finally, we recall the notion of commutative and cocommutative
differential ASI bialgebras, reformulated from the original form introduced in~\mcite{LLB}.

\begin{defi} \mcite{LLB}
A quintuple $(A,\cdot, \delta,D,Q)$ is called a commutative and cocommutative {\bf differential antisymmetric infinitesimal bialgebra} or simply a commutative and cocommutative  {\bf differential ASI bialgebra} if
\begin{enumerate}
    \item $(A,\cdot,\delta)$ is an ASI bialgebra,
    \item  $(A,\cdot, D, Q)$ is an
    admissible commutative differential algebra, and
    \item  $(A,\delta, Q, D)$ is an admissible cocommutative differential coalgebra.
\end{enumerate}
\mlabel{de:admdifbialg}
\vspb
\end{defi}

\subsection{Construction of Novikov bialgebras from commutative and cocommutative differential ASI bialgebras}
Now we construct Novikov bialgebras from commutative and cocommutative differential ASI bialgebras.

\begin{pro}\mlabel{prostr1}
Let $(A, \cdot, \delta, D, Q)$ be a commutative and cocommutative
differential ASI bialgebra and $q\in {\bf k}$. Define a binary
operation $\circ_q$ and a coproduct $\Delta_q$ on $A$ by
Eqs.~\meqref{eq:k} and \meqref{eq:cons2} respectively. Then $(A,
\circ_q, \Delta_q)$ is a Novikov bialgebra if and only if, using the notation $\delta(a)=\sum a_{(1)}\otimes
a_{(2)}$, the following equalities hold for all $a, b\in A$.
{\wuhao
\begin{equation}
\begin{split}
&\sum\hspace{-.1cm} a_{(1)}\otimes \Big((q^2-q-1)D(b)\!\cdot\! (Q+D)a_{(2)}\!-\!(D+Q)(b)\!\cdot\! Q(a_{(2)})+q^2(D(b)\!\cdot\! Q(a_{(2)})-D(a_{(2)})\!\cdot\! Q(b))\\
&\quad+(q^2-2q-1)b\!\cdot\! D(D+Q)a_{(2)}+(q^2-q)(b\!\cdot\! DQ(a_{(2)})-b\!\cdot\! QD(a_{(2)}))-2q b\!\cdot\! Q(Q+D)a_{(2)}\Big) \\
&\quad+(1-2q^2+q)\sum D(a_{(1)})\otimes b\cdot (Q+D)a_{(2)}=0,
\end{split}
\mlabel{eq:bialg1}
\end{equation}
}
{\wuhao
\begin{equation}
    \begin{split}
&(1+2q)\sum\Big((D+Q)(a)\cdot b_{(1)}\otimes (Q+qD)b_{(2)}-(Q+qD)b_{(1)}\otimes (D+Q)(a)\cdot b_{(2)}\Big)\\
&\quad=(1+2q)\sum\Big((D+Q)(b)\cdot a_{(1)}\otimes (Q+qD)a_{(2)}-(Q+qD)a_{(1)}\otimes (D+Q)(b)\cdot a_{(2)}\Big),
\end{split}
\mlabel{eq:bialg2}
\end{equation}}
{\wuhao\begin{equation}
\begin{split}
&(1+2q)(\id \otimes L_{\cdot}((D+qQ)a)-L_{\cdot}((D+qQ)a)\otimes \id)\sum b_{(1)}\otimes (D+Q)b_{(2)}\\
&\quad=(1+2q)(\id \otimes L_{\cdot}((D+qQ)b)-L_{\cdot}((D+qQ)b)\otimes \id)\sum a_{(1)}\otimes (D+Q)a_{(2)}.
\end{split}
\mlabel{eq:bialg3}
\end{equation}}
\end{pro}

\begin{proof}
By Propositions \mref{constr2} and \mref{pro:co-cons},
$(A,\circ_q)$ is a Novikov algebra and $(A, \Delta_q)$ is a
Novikov coalgebra. Let $a,b\in A$.
Then we have
{\wuhao
\begin{eqnarray*}
&&\Delta_q(a\circ_{q} b)-(R_{\circ_{q}}(b)\otimes \id)\Delta_q(a)-(\id\otimes L_{\star_q}(a))(\Delta_q(b)+\tau\Delta_q(b))\\
&&\quad=\Delta_q(a\cdot (D+qQ)b)-(R_{\circ_q}(b)\otimes \id)\Big(\sum a_{(1)}\otimes(Q+qD) a_{(2)}\Big)\\
&&\qquad-(\id\otimes L_{\star_q}(a))\Big(\sum b_{(1)}\otimes (Q+qD)b_{(2)}+(Q+qD)b_{(1)}\otimes b_{(2)}\Big)\\
&&\quad=(\id\otimes (Q+qD))\delta(a\cdot (D+qQ)b)-\sum a_{(1)}\circ_{q} b\otimes (Q+qD)a_{(2)}\\
&&\qquad-\sum b_{(1)}\otimes (a\circ_{q}(Q+qD)b_{(2)}+((Q+qD)b_{(2)})\circ_{q}a)\\
&&\qquad-\sum (Q+qD)b_{(1)}\otimes(a\circ_{q}b_{(2)}+b_{(2)}\circ_{q}a)\\
&&\quad=(\id\otimes (Q+qD))((\id\otimes L_{\cdot}(a))\delta((D+qQ)b)+(L_{\cdot}((D+qQ)b)\otimes \id)\delta(a))\\
&&\qquad-\sum a_{(1)}\cdot (D+qQ)b\otimes (Q+qD)a_{(2)}-\sum b_{(1)}\otimes a\cdot (D+qQ)(Q+qD)b_{(2)}\\
&&\qquad-\sum b_{(1)}\otimes(Q+qD)b_{(2)}\cdot (D+qQ)a-\sum (Q+qD)b_{(1)}\otimes (a\cdot (D+qQ)b_{(2)}+b_{(2)}\cdot (D+qQ)a)\\
&&\quad=(\id\otimes (Q+qD))((\id\otimes L_{\cdot}(a))((D\otimes \id-\id\otimes Q)\delta(b)+q(\id\otimes Q+Q\otimes \id)\delta(b)))\\
&&\qquad+\sum (D+qQ)(b)\cdot a_{(1)}\otimes(Q+qD) a_{(2)}-\sum a_{(1)}\cdot (D+qQ)b\otimes (Q+qD)a_{(2)}\\
&&\qquad-\sum b_{(1)}\otimes a\cdot (D+qQ)(Q+qD)b_{(2)}-\sum b_{(1)}\otimes ((Q+qD)b_{(2)})\cdot (D+qQ)a\\
&&\qquad-\sum (Q+qD)b_{(1)}\otimes (a\cdot (D+qQ)b_{(2)}+b_{(2)}\cdot (D+qQ)a)\\
&&\quad=\sum b_{(1)}\otimes (-(Q+qD)(a\cdot Q(b_{(2)}))+q(Q+qD)(a\cdot Q(b_{(2)}))\\
&&\qquad-a\cdot (D+qQ)(Q+qD)b_{(2)}-((Q+qD)b_{(2)})\cdot (D+qQ)a)\\
&&\qquad+\sum D(b_{(1)})\otimes ((Q+qD)(a\cdot b_{(2)})-qa\cdot (D+qQ)b_{(2)}-qb_{(2)}\cdot (D+qQ)a)\\
&&\qquad+\sum Q(b_{(1)})\otimes (q(Q+qD)(a\cdot b_{(2)})-a\cdot (D+qQ)b_{(2)}-b_{(2)}\cdot (D+qQ)a).
\end{eqnarray*}
}
One can directly check that
{\wuhao \begin{eqnarray*}
&&(Q+qD)(a\cdot b_{(2)})-qa\cdot (D+qQ)b_{(2)}-qb_{(2)}\cdot (D+qQ)a\\
&&\quad=(1-q^2)Q(a)\cdot b_{(2)}-a\cdot D(b_{(2)})-q^2a\cdot Q(b_{(2)}),\\
&&q(Q+qD)(a\cdot b_{(2)})-a\cdot (D+qQ)b_{(2)}-b_{(2)}\cdot (D+qQ)a\\
&&\quad =(q^2-q-1)a\cdot D(b_{(2)})+(q^2-1)D(a)\cdot b_{(2)}-qa\cdot Q(b_{(2)}),\\
\end{eqnarray*}}
and
{\wuhao\begin{eqnarray*}
&&(1-q^2)Q(a)\cdot b_{(2)}-a\cdot D(b_{(2)})-q^2a\cdot Q(b_{(2)})\\
&&\qquad-((q^2-q-1)a\cdot D(b_{(2)})+(q^2-1)D(a)\cdot b_{(2)}-qa\cdot Q(b_{(2)}))\\
&&\quad=(1-2q^2+q)a\cdot (Q+D)b_{(2)}.
\end{eqnarray*}}
Note that $(D\otimes \id-\id\otimes Q)\delta(b)=(\id\otimes D-Q\otimes \id)\delta(b)$. Therefore, we get
\begin{eqnarray}\mlabel{eq:adm-cocond}
\sum (D+Q)b_{(1)}\otimes b_{(2)}=\sum b_{(1)}\otimes (D+Q)b_{(2)}.
\end{eqnarray}
Then one obtains
{\wuhao
\begin{eqnarray*}
&&\Delta_q(a\circ_{q} b)-(R_{\circ_{q}}(b)\otimes \id)\Delta_q(a)-(\id\otimes L_{\star_q})(\Delta_q(b)+\tau\Delta_q(b))\\
&&\quad=\sum b_{(1)}\otimes (-(Q+qD)(a\cdot Q(b_{(2)}))+q(Q+qD)(a\cdot Q(b_{(2)}))\\
&&\qquad-a\cdot (D+qQ)(Q+qD)b_{(2)}-((Q+qD)b_{(2)})\cdot (D+qQ)a)\\
&&\qquad+\sum (D+Q)(b_{(1)})\otimes ((q^2-q-1)a\cdot D(b_{(2)})+(q^2-1)D(a)\cdot b_{(2)}-qa\cdot Q(b_{(2)}))\\
&&\qquad+(1-2q^2+q)\sum D(b_{(1)})\otimes a\cdot (Q+D)b_{(2)}\\
&&\quad= \sum b_{(1)}\otimes (-(Q+qD)(a\cdot Q(b_{(2)}))+q(Q+qD)(a\cdot Q(b_{(2)}))-a\cdot (D+qQ)(Q+qD)b_{(2)}\\
&&\qquad-((Q+qD)b_{(2)})\cdot (D+qQ)a +(q^2-q-1)a\cdot D(D+Q)b_{(2)}+(q^2-1)D(a)\cdot (D+Q)b_{(2)}\\
&&\qquad-qa\cdot Q(D+Q)b_{(2)})+(1-2q^2+q)\sum D(b_{(1)})\otimes a\cdot (Q+D)b_{(2)}\\
&&\quad=\sum b_{(1)}\otimes ((q^2-q-1)D(a)\cdot (Q+D)b_{(2)}-(D+Q)a\cdot Q(b_{(2)})+q^2(D(a)\cdot Q(b_{(2)})-D(b_{(2)})\cdot Q(a))\\
&&\qquad+(q^2-2q-1)a\cdot D(D+Q)(b_{(2)})+(q^2-q)(a\cdot DQ(b_{(2)})-a\cdot QD(b_{(2)}))-2q a\cdot Q(Q+D)b_{(2)})\\
&&\qquad+(1-2q^2+q)\sum D(b_{(1)})\otimes a\cdot (Q+D)b_{(2)}\\
&&\quad=0.
\end{eqnarray*}
}
Therefore, Eq. (\mref{Lb5}) holds if and only if Eq.
(\mref{eq:bialg1}) holds.

Similarly, we have
\vspb
{\wuhao \begin{eqnarray*}
&&(L_{ \star_q}(a)\otimes \id)\Delta_q(b)-(\id\otimes L_{\star_q}(a))\tau \Delta_q(b)\\
&&\quad=\sum ((a\cdot (D+qQ)b_{(1)}+b_{(1)}\cdot (D+qQ)a)\otimes (Q+qD)b_{(2)}\\
&&\qquad-(Q+qD)b_{(1)}\otimes (a\cdot (D+qQ)b_{(2)}+b_{(2)}\cdot (D+qQ)a))\\
&&\quad=\sum (-(Q+qD)(a\cdot b_{(1)})+(q+1)(D+Q)(a)\cdot b_{(1)}+qa\cdot (D+Q)b_{(1)})\otimes (Q+qD)b_{(2)}\\
&&\qquad-\sum (Q+qD)b_{(1)}\otimes (-(Q+qD)(a\cdot b_{(2)})+(q+1)(D+Q)(a)\cdot b_{(2)}+qa\cdot (D+Q)b_{(2)}))\\
&&\quad=\sum (-(Q+qD)(a\cdot b_{(1)})+(2q+1)(D+Q)(a)\cdot b_{(1)})\otimes (Q+qD)b_{(2)}\\
&&\qquad-\sum (Q+qD)b_{(1)}\otimes (-(Q+qD)(a\cdot b_{(2)})+(2q+1)(D+Q)(a)\cdot b_{(2)})).
\end{eqnarray*}}
Since $\sum a_{(1)}\cdot b\otimes a_{(2)}-a_{(1)}\otimes
b\cdot a_{(2)}=\sum b_{(1)}\cdot a\otimes b_{(2)}-b_{(1)}\otimes
a\cdot b_{(2)}$, we have
{\wuhao\begin{eqnarray*}
&&-\sum (Q+qD)(a\cdot b_{(1)})\otimes (Q+qD)b_{(2)}+\sum (Q+qD)b_{(1)}\otimes (Q+qD)(a\cdot b_{(2)})\\
&&\quad=-\sum (Q+qD)(b\cdot a_{(1)})\otimes (Q+qD)a_{(2)}+\sum (Q+qD)a_{(1)}\otimes (Q+qD)(b\cdot a_{(2)}).
\end{eqnarray*}}
Therefore, Eq. (\mref{Lb6}) holds if and only if Eq. (\mref{eq:bialg2}) holds.

Note that
\vspb
{\wuhao\begin{eqnarray*}
\Delta_q(b)+\tau\Delta_q(b)&=& \sum (b_{(1)}\otimes (Q+qD)b_{(2)}+(Q+qD)b_{(1)}\otimes b_{(2)})\\
&=& -\delta((D+qQ)b)+q\sum b_{(1)}\otimes (D+Q)b_{(2)}+(q+1)\sum (D+Q)b_{(1)}\otimes b_{(2)}\\
&=& -\delta((D+qQ)b)+(2q+1)\sum b_{(1)}\otimes (D+Q)b_{(2)}.
\end{eqnarray*}}
Therefore,
\vspb
{\wuhao\begin{eqnarray*}
&&(\id\otimes R_{\circ_{q}}(a)-R_{\circ_{q}}(a)\otimes \id)(\Delta_q(b)+\tau\Delta_q(b))\\
&&\quad=(\id\otimes L_{\cdot}((D+qQ)a)-L_{\cdot}((D+qQ)a)\otimes \id)(-\delta((D+qQ)b)+(2q+1)\sum b_{(1)}\otimes (D+Q)b_{(2)}).
\end{eqnarray*}}
By Eq.~(\mref{ASI2}), we conclude that Eq. (\mref{Lb7}) holds if and only if Eq. (\mref{eq:bialg3})
holds.

Now the proof is completed.
\vspb
\end{proof}

\begin{thm}
Let $(A,\cdot, \delta,D,Q)$ be a commutative and cocommutative
differential ASI bialgebra. Let $q\in \bfk$ be given. Define a binary operation $\circ_q$ and a
coproduct $\Delta_q$ on $A$ by Eqs.~\meqref{eq:k} and \meqref{eq:cons2} respectively.
\begin{enumerate}
\item \label{it:aaa} When $q=-\frac{1}{2}$, then $(A,
\circ_{-\frac{1}{2}}, \Delta_{-\frac{1}{2}})$ is a Novikov
bialgebra without any condition. \mlabel{it:const2}
\item\label{it:bbb}  For other $q\in \bfk$, if
\vspb
\begin{eqnarray}
&a\cdot Q(b)=-a\cdot D(b),\;\;a, b\in A,
\mlabel{eq:conda}
\\
&(\id\otimes
Q)\delta =-(\id \otimes D)\delta,  \mlabel{eq:condb}
\end{eqnarray}
then $(A, \circ_q, \Delta_q)$ is a Novikov bialgebra.
In particular,
if $Q=-D$, then $(A, \circ_q, \Delta_q)$ is a Novikov bialgebra.
\mlabel{it:const1}
\end{enumerate}
\mlabel{bialgebra-constr}
\end{thm}

\begin{proof}
By Proposition \mref{prostr1}, we only need to check that Eqs. (\mref{eq:bialg1}), (\mref{eq:bialg2}) and (\mref{eq:bialg3}) hold.

\noindent
\meqref{it:const2} Obviously, Eqs. (\mref{eq:bialg2}) and (\mref{eq:bialg3}) hold when $q=-\frac{1}{2}$.
Also note that
{\wuhao\begin{eqnarray*}
    &&-\frac{1}{4}D(b)\cdot (D+Q)a_{(2)}-(D+Q)(b)\cdot Q(a_{(2)})+\frac{1}{4}(D(b)\cdot Q(a_{(2)})-D(a_{(2)})\cdot Q(b))\\
    &&\quad +\frac{1}{4}b\cdot D(D+Q)a_{(2)}+\frac{3}{4}(b\cdot DQ(a_{(2)})-b\cdot QD(a_{(2)}))+b\cdot Q(D+Q)a_{(2)}\\
    &=&-\frac{1}{4}(D+Q)(b)\cdot D(a_{(2)})-(D+Q)(b)\cdot Q(a_{(2)})+\frac{1}{4}b\cdot D^2(a_{(2)})\\
    &&\quad +b\cdot DQ(a_{(2)})-\frac{3}{4}b\cdot QD(a_{(2)})+b\cdot Q(D+Q)a_{(2)}\\
    &=&-\frac{1}{4}b\cdot (D+Q)D(a_{(2)})-b\cdot (D+Q)Q(a_{(2)})+\frac{1}{4}b\cdot D^2(a_{(2)})\\
    &&\quad +b\cdot DQ(a_{(2)})-\frac{3}{4}b\cdot QD(a_{(2)})+b\cdot Q(D+Q)a_{(2)}\\
    &=&0.
\end{eqnarray*}}
Then Eq. (\mref{eq:bialg1}) holds. Therefore, $(A, \circ_{-\frac{1}{2}}, \Delta_{-\frac{1}{2}})$ is a Novikov bialgebra.

\smallskip

\noindent
\meqref{it:const1} Note that in this case we have
\begin{eqnarray}
\mlabel{eq:casea1}a\cdot (D+Q)b=0,\;\;\;\sum a_{(1)}\otimes (D+Q)a_{(2)}=0, \;\; a, b\in A,
 \end{eqnarray}
where $\delta(a)=\sum a_{(1)}\otimes a_{(2)}$. 

It is easy to see that Eqs. (\mref{eq:bialg2}) and
(\mref{eq:bialg3}) hold.
 Moreover,  by Eq. (\mref{eq:casea1}), one can directly obtain that Eq. (\mref{eq:bialg1})
 holds if and only if
\begin{eqnarray}
&&\mlabel{eq:casea2}\sum a_{(1)}\otimes (q^2(D(b)\cdot
Q(a_{(2)})-D(a_{(2)})\cdot Q(b))+(q^2-q)(b\cdot DQ(a_{(2)})-b\cdot
QD(a_{(2)})))=0,\;\;a,b\in A.
\end{eqnarray}
Note that
\begin{eqnarray*}
&&D(b)\cdot Q(a_{(2)})-D(a_{(2)})\cdot Q(b)=D(b)\cdot Q(a_{(2)})+Q(a_{(2)})\cdot Q(b)=((D+Q)b)\cdot Q(a_{(2)})=0,\\
&&b\cdot DQ(a_{(2)})-b\cdot QD(a_{(2)})=b\cdot DQ(a_{(2)})+b\cdot
D^2(a_{(2)})=b\cdot D(D+Q)a_{(2)}.
\end{eqnarray*}
Then the left hand side of Eq. (\mref{eq:casea2}) becomes
$$ (q^2-q) \sum a_{(1)}\otimes b\cdot D(D+Q)a_{(2)}=(q^2-q)(\id \otimes L_\cdot(b))(\id\otimes D)\Big(\sum  a_{(1)}\otimes (D+Q)a_{(2)}\Big).$$
Hence Eq. (\mref{eq:casea2}) follows from
Eq.~\meqref{eq:casea1}. Therefore, $(A, \circ_q, \Delta_q)$ is a
Novikov bialgebra.

In particular, if $Q=-D$, then Eqs. (\mref{eq:conda}) and (\mref{eq:condb}) naturally hold. Therefore, $(A, \circ_q, \Delta_q)$ is a Novikov bialgebra.
\end{proof}

\begin{rmk}
By Theorem \mref{bialgebra-constr} \meqref{it:const1}, if Eqs. (\mref{eq:conda}) and (\mref{eq:condb}) are satisfied, then $(A, \circ_0,\Delta_0)$ is a Novikov bialgebra. On the other hand, we can define an infinitesimal deformation of a Novikov bialgebra as follows.
Let $(A,\circ,\Delta)$ be a Novikov bialgebra. Define a binary operation $\circ_q$ and a coproduct $\Delta_q$ for all $q\in \bfk$ by Eqs.~(\mref{eq:deform}) and (\mref{eq:deformc}) respectively through the linear maps $f:A\otimes A\rightarrow A$
and $\Theta :A\rightarrow A\otimes A$. Then $(A,\circ_q,\Delta_q)$ is called an {\bf infinitesimal deformation} of $(A,\circ, \Delta)$ if $(A, \circ_q, \Delta_q)$ is a Novikov bialgebra for all $q\in {\bf k}$.
Therefore, in this sense, $(A, \circ_q,\Delta_q)$ given in Theorem \mref{bialgebra-constr} \meqref{it:const1} is an infinitesimal deformation of the Novikov bialgebra  $(A, \circ_0,\Delta_0)$.
\end{rmk}

As introduced in~\cite{LLB}, a {\bf derivation} on an ASI bialgebra $(A,\cdot, \delta)$ is defined to be a linear operator on $A$ that is both a derivation on
$(A,\cdot)$ and a coderivation on $(A,\delta)$.

\begin{pro}\label{pro:cond}
\begin{enumerate}
\item \label{itemd1}Let $(A,\cdot,D,Q)$ be an admissible
commutative differential algebra. Then $Q$ is a derivation on
$(A,\cdot)$ if and only if  Eq.~\meqref{eq:conda} holds. \item
\label{itemd2}Let $(A,\delta,Q,D)$ be an
admissible cocommutative differential coalgebra. Then $D$ is a
coderivation on the coalgebra $(A,\delta)$ if and only if Eq.~
\meqref{eq:condb} holds. \item \label{itemd3}Let $(A,\cdot,
\delta,D,Q)$ be a commutative and    cocommutative differential
ASI bialgebra. Then $D$ and $Q$ are derivations on the ASI
bialgebra $(A,\cdot, \delta)$ if and only if
Eqs.~\meqref{eq:conda} and \meqref{eq:condb} hold.
\end{enumerate}
\mlabel{pp:bialgder0}
\end{pro}

\begin{proof}
(\mref{itemd1}) Since $(A,\cdot,D,Q)$ is an admissible commutative differential algebra,
Eq. (\mref{eq:RSI1}) holds. Note that $Q$ is a derivation on $(A,\cdot)$ if and only if $Q(a\cdot b)=Q(a)\cdot b+a\cdot Q(b)$ for all $a$, $b\in A$. Therefore, by Eq. (\mref{eq:RSI1}), $Q$ is a derivation on $(A,\cdot)$ if and only if  Eq.~\meqref{eq:conda} holds.

Items (\mref{itemd2}) and (\mref{itemd3}) can be proved similarly.
\end{proof}

Now Theorem \mref{bialgebra-constr} \meqref{it:const1} can be restated as follows.
\begin{cor}
Let $(A,\cdot, \delta,D,Q)$ be a commutative and    cocommutative differential ASI bialgebra.
If both $D$ and $Q$ are derivations on the
ASI bialgebra $(A,\cdot, \delta)$, then $(A, \circ_q, \Delta_q)$
is a Novikov bialgebra for all $q\in {\bf k}$, where $\circ_q$ and $\Delta_q$ are defined by Eqs.~\meqref{eq:k} and \meqref{eq:cons2} respectively.
\mlabel{pp:bialgder}
\end{cor}

In the following, we apply Theorem~\mref{bialgebra-constr} to construct two Novikov bialgebras.

\begin{ex}\mlabel{ex:Nov1}
Let $(A={\bf k}e_1\oplus {\bf k}e_2, \cdot)$ be the commutative
associative algebra whose
product is given by
\begin{eqnarray*}
e_1\cdot e_1=e_1,\;\;e_1\cdot e_2=e_2\cdot e_1=e_2,\;\;e_2\cdot e_2=0.
\end{eqnarray*}
Let $D$, $Q: A\rightarrow A$ be the linear maps given by
\begin{eqnarray*}
&&D(e_1)=0,\;\;D(e_2)=e_2,\;\;Q(e_1)=e_1,\;\;Q(e_2)=0.
\end{eqnarray*}
Define $\delta: A\rightarrow A\otimes A$ to be the linear map given by
\begin{eqnarray*}
\delta(e_1)=0,\;\;\delta(e_2)=e_2\otimes e_2.
\end{eqnarray*}
One checks that $(A, \cdot, \delta, D, Q)$ is a commutative and cocommutative differential ASI bialgebra.
Let $(A, \circ_{-\frac{1}{2}})$ be the Novikov algebra induced from $(A, \cdot, D, Q)$ whose product is given by
\begin{eqnarray*}
e_1\circ_{-\frac{1}{2}}e_1=-\frac{1}{2}e_1,\;\;e_1\circ_{-\frac{1}{2}}e_2=e_2,\;\;e_2\circ_{-\frac{1}{2}}e_1=-\frac{1}{2}e_2,\;\;e_2\circ_{-\frac{1}{2}}e_2=0.
\end{eqnarray*}
Let $(A, \Delta_{-\frac{1}{2}})$ be the Novikov coalgebra induced from $(A,\delta, Q, D)$ whose coproduct is given by
\begin{eqnarray*}
&&\Delta_{-\frac{1}{2}}(e_1)=0,\;\;\Delta_{-\frac{1}{2}}(e_2)=-\frac{1}{2}e_2\otimes e_2.
\end{eqnarray*}
By Theorem \mref{bialgebra-constr}, $(A, \circ_{-\frac{1}{2}}, \Delta_{-\frac{1}{2}})$ is a Novikov bialgebra.
\end{ex}

\begin{ex}\mlabel{conex1}
Let $({\bf k}[x],\cdot)$ be the polynomial algebra. Define a linear map $\delta: {\bf k}[x]\rightarrow {\bf k}[x]\otimes {\bf k}[x]$ by
\begin{eqnarray*}
\delta(1)=0,\;\;\delta(x^n)=x^{n-1}\otimes 1+x^{n-2}\otimes x+\cdots+x\otimes x^{n-2}+1\otimes x^{n-1}, \;\;\;\; n\geq 1.
\end{eqnarray*}
By  \cite[Example 2.3]{A1}, $({\bf k}[x], \cdot, \delta)$ is a commutative and cocommutative ASI bialgebra. Let $D=\frac{d}{dx}$. Obviously, $D$ is also a coderivation of $({\bf k}[x], \delta)$. Therefore, $({\bf k}[x], \cdot, \delta, D, -D)$ is a commutative and cocommutative differential ASI bialgebra. Then by Theorem \mref{bialgebra-constr},
there is a Novikov bialgebra $({\bf k}[x], \circ_q, \Delta_q)$ given by
{\wuhao\begin{eqnarray*}
&&x^m\circ_q x^n=(1-q)x^m \cdot D(x^n)=(1-q)nx^{m+n-1},\;\;\Delta_q(1)=(q-1)(\id\otimes D)\delta(1)=0, \;\;m, n\geq 0,\\
&&\Delta_q(x)=(q-1)(\id \otimes D)\delta(x)=0,\;\;\Delta_q(x^n)=(q-1)(\id\otimes D)\delta(x^n)=(q-1)\sum_{i=1}^{n-1}ix^{n-1-i}\otimes x^{i-1}, \;\; n\geq 2.
\end{eqnarray*}}
\end{ex}

Finally, we give a construction of infinite-dimensional Lie bialgebras from commutative and cocommutative differential ASI bialgebras.
First recall the construction of infinite-dimensional Lie bialgebras from Novikov bialgebras, following the notations in \mcite{HBG}.

Let $C=\bigoplus_{i\in \ZZ}C_i$ and $D=\bigoplus_{i\in \ZZ}D_i$ be $\ZZ$-graded vector spaces with all $C_i$ and $D_i$ finite-dimensional. Define the {\bf completed tensor product} of $C$ and $D$ to be the vector space
\begin{eqnarray*}
C\widehat{\otimes}D:=\prod_{i,j\in \ZZ}C_i\otimes D_j.
\vspb
\end{eqnarray*}
Thus a general term of $C\hatot  D$ is a possibly infinite sum
\vspace{-.2cm}
\begin{equation}\mlabel{eq:ssum}
 \sum_{i,j,\alpha} c_{i\alpha}\ot d_{j\alpha},
 \vspb
 \end{equation}
where $i,j\in \ZZ$ and $\alpha$ is in a finite index set (which might depend on $i$ and $j$).

With these notations, for linear maps $f:C\to C'$ and $g:D\to D'$, define
\vspb
$$ f\hatot g: C\hatot D \to C' \hatot D',
\quad \sum_{i,j,\alpha} c_{i\alpha}\ot d_{j\alpha}\mapsto  \sum_{i,j,\alpha} f(c_{i\alpha})\ot g(d_{j\alpha}).
\vspace{-.2cm}
$$
Also the twist map $\tau$ has its completion
\vspb
$$ \widehat{\tau}: C\hatot C \to C\hatot C, \quad
\sum_{i,j,\alpha} c_{i\alpha} \ot d_{j\alpha} \mapsto \sum_{i,j,\alpha} d_{j\alpha}\ot c_{i\alpha}.
\vspb
$$

\begin{defi}\mcite{HBG}
A {\bf \complete Lie coalgebra}
 is a pair $(L, \delta)$, where
$L=\oplus_{i\in \ZZ} L_i$
  is a $\ZZ$-graded vector space with all $L_i$ finite-dimensional and $\delta: L\rightarrow L\hatot L$ is a linear map satisfying
    \begin{eqnarray}
        \mlabel{lia1} \notag
        &&\delta(a)=-\widehat{\tau} \delta(a), \\
        \mlabel{lia2} \notag
        &&(\id\hatot \delta)\delta(a)-(\widehat{\tau}\hatot  \id)(\id\hatot  \delta)\delta(a)=(\delta\hatot  \id)\delta(a),\;\;\; a\in L.
    \end{eqnarray}

A {\bf \complete Lie bialgebra}
is a triple $(L, [\cdot,\cdot],\delta)$ such that
$(L, [\cdot, \cdot])$ is a Lie algebra, $(L,
\delta)$ is a \complete Lie coalgebra, and the following
compatibility condition holds.
\vspa
\begin{eqnarray} \mlabel{lia3}
\notag \delta([a,b])=(\ad_a\hatot  \id+\id \hatot
\ad_a)\delta(b)-(\ad_b\hatot  \id+\id \hatot
\ad_b)\delta(a),\;\;\;\;  a, b\in L,
\vspb
\end{eqnarray}
where $\ad_a (b)=[a,b]$ for all $a, b\in
L$.
\end{defi}
\begin{pro}\mlabel{Liebi-constr}\cite[Theorem 2.23]{HBG}
Let $(A, \circ, \Delta)$ be a Novikov bialgebra and $L=A\otimes {\bf k}[t,t^{-1}]$ be the Lie algebra defined by
\vspb
\begin{eqnarray*}
[at^m,bt^n]=m(a\circ b)t^{m+n-1}-n(b\circ a)t^{m+n-1}, \;\;\; a, b\in A, m, n\in \ZZ,
\end{eqnarray*}
where $at^m:=a\otimes t^m$. Define the linear map $\delta_L:L\rightarrow L\widehat{\otimes} L$ by
\begin{eqnarray}
\delta_L(at^m)=\sum_{i\in \ZZ}\sum (i+1)(a_{(1)}t^{-i-2}\otimes a_{(2)}t^{m+i}-a_{(2)}t^{m+i}\otimes a_{(1)}t^{-i-2}),\;\;a\in A, m\in \ZZ,
\end{eqnarray}
where $\Delta(a)=\sum a_{(1)}\otimes a_{(2)}$. Then $(L, [\cdot,\cdot], \delta_L)$ is a completed Lie bialgebra.
\end{pro}

\begin{thm}\mlabel{thm-con-Lie-bi}
Let $(A,\cdot, \delta,D,Q)$ be a commutative and cocommutative
differential ASI bialgebra and $L=A\otimes {\bf k}[t,t^{-1}]$.
Define a binary operation on $L$ for $q\in {\bf k}$ as
follows.
 \begin{eqnarray}
 [at^m,bt^n]_q=m(a\cdot (D+qQ) b)t^{m+n-1}-n(b\cdot (D+qQ)a)t^{m+n-1}, \;\;\; a, b\in A, m, n\in \ZZ.
 \end{eqnarray}
 Let $\delta_{L,q}: L\rightarrow L\widehat{\otimes} L$ be the linear map defined by
 \begin{eqnarray}
\delta_{L,q}(at^m)=\sum_{i\in \ZZ}\sum (i+1)(a_{(1)}t^{-i-2}\otimes (Q+qD)a_{(2)}t^{m+i}-(Q+qD)a_{(2)}t^{m+i}\otimes a_{(1)}t^{-i-2}),
\vspb
\end{eqnarray}
for all $a\in A$ and $m\in \ZZ$, where $\delta(a)=\sum
a_{(1)}\otimes a_{(2)}$. Then we have the following
conclusions.
\begin{enumerate}
\item $(L, [\cdot,\cdot]_{-\frac{1}{2}}, \delta_{L,-\frac{1}{2}})$ is a completed Lie bialgebra.
\item If Eqs.~\meqref{eq:conda} and \meqref{eq:condb} hold,
then $(L, [\cdot,\cdot]_q, \delta_{L,q})$ is a completed Lie bialgebra for any $q\in {\bf k}$.
\end{enumerate}
\end{thm}

\begin{proof}
The conclusions are direct consequences of Theorem \mref{bialgebra-constr} and Proposition \mref{Liebi-constr}.
\end{proof}

\begin{ex}
Let $(A, \cdot, \delta, D, Q)$ be the commutative and cocommutative differential ASI bialgebra given in Example \mref{ex:Nov1}. Then by Theorem \mref{thm-con-Lie-bi}, there is a completed Lie bialgebra $(L=A\otimes {\bf k}[t,t^{-1}], [\cdot,\cdot]_{-\frac{1}{2}}, \delta_{L,-\frac{1}{2}})$ given by
\begin{eqnarray*}
&[e_1t^m,e_1t^n]_{-\frac{1}{2}}=-\frac{1}{2}(m-n)e_1t^{m+n-1},\;\;[e_1t^m,e_2t^n]_{-\frac{1}{2}}=(m+\frac{n}{2})e_2t^{m+n-1},\;\;[e_2t^m,e_2t^n]_{-\frac{1}{2}}=0,\;\;m, n\in \ZZ,&\\
&\delta_{L,-\frac{1}{2}}(e_1t^m)=0,\;\;\delta_{L,-\frac{1}{2}}(e_2t^m)=\sum_{i\in \ZZ}(-i-1-\frac{m}{2})e_2t^{-i-2}\otimes e_2t^{m+i},\;\; m\in \ZZ.&
\end{eqnarray*}
\end{ex}

\subsection{Characterizations in terms of Manin triples }
We recall the definition of a quadratic Novikov algebra.
\begin{defi}\mcite{HBG}\mlabel{Novbilinear}
Let $(A,\circ)$ be a Novikov algebra. A bilinear form
$\mathcal{B}(\cdot,\cdot)$  on $A$ is called {\bf invariant} if it satisfies
\vspb
\begin{eqnarray}\mlabel{bilinear1}
\mathcal{B}(a\circ b,c)=-\mathcal{B}(b, a\star c),\;\;\;
a,b,c\in A.
\end{eqnarray}
A {\bf quadratic Novikov algebra}, denoted by $(A,
\circ,\mathcal{B}(\cdot,\cdot))$, is a Novikov algebra $(A,\circ)$
together with a nondegenerate symmetric invariant bilinear form
$\mathcal{B}(\cdot,\cdot)$.
\end{defi}

Recall that a bilinear form $\mathcal B(\cdot,\cdot)$ on an associative algebra $(A, \cdot)$ is {\bf
invariant} if
\begin{eqnarray*}
\mathcal B(a \cdot b,c)=\mathcal B(a, b\cdot c),\;\;\; a, b, c\in A.
\end{eqnarray*}
A {\bf commutative Frobenius algebra} $(A, \cdot, \mathcal B(\cdot,\cdot))$ is a commutative associative algebra $(A, \cdot)$ with a non-degenerate invariant bilinear form $\mathcal B(\cdot,\cdot)$. A commutative Frobenius algebra $(A, \cdot, \mathcal B(\cdot,\cdot))$ is called {\bf symmetric} if $\mathcal B(\cdot,\cdot)$ is symmetric.

\begin{defi}\mcite{LLB}
Let $(A, \cdot, D)$ be a commutative differential algebra. A {\bf commutative differential Frobenius algebra} on $(A, \cdot, D)$ is $(A, \cdot, D, \mathcal B(\cdot,\cdot))$ such that $(A, \cdot, \mathcal B(\cdot,\cdot))$ is a commutative symmetric Frobenius algebra.
\end{defi}

\begin{lem} \cite[Proposition 3.3]{LLB}\mlabel{admissible}
Let $(A,\cdot, D, \mathcal B(\cdot,\cdot))$ be a commutative differential Frobenius algebra.
Suppose that $\hat D$ is
the adjoint operator of $D$ with respect to $\mathcal
B(\cdot,\cdot)$ in the sense that
\begin{equation}
\mathcal B(D(a),b)=\mathcal B(a, \hat D(b)),\;\;\;\;\; a,b\in A.
\end{equation}
Then $(A, \cdot, D, \hat D )$ is an admissible commutative differential algebra.
\end{lem}

Next, we present a construction of quadratic Novikov algebras from commutative differential Frobenius algebras.

\begin{pro}\mlabel{cor:constr}
Let $(A,\cdot, D, \mathcal B(\cdot,\cdot))$ be a commutative differential Frobenius algebra.
Let $\hat D$ be the adjoint operator of $D$ with respect to $\mathcal B(\cdot,\cdot)$. Define $a\circ_q b:=a\cdot (D+q \hat D)b$ for all $a$, $b\in A$. Then $(A,\circ_q,\mathcal B(\cdot,\cdot) )$ is a quadratic Novikov
algebra if and only if $q=-\frac{1}{2}$ or $\hat D$ is a derivation of $(A,\cdot)$.
\end{pro}

\begin{proof}
By Lemma \mref{admissible}, $(A, \cdot, D, \hat D)$ is an
admissible commutative differential algebra. Hence by
Proposition~\ref{constr2}, $(A,\circ_q)$ is a Novikov algebra for
any $q\in {\bf k}$. Furthermore, for all $a,b,c\in A$, we have
{\wuhao\begin{eqnarray*}
&&\mathfrak{B}(a\circ_{q} b,c)+\mathfrak{B}(b, a\circ_q c+c\circ_q a)\\
&=&\mathfrak{B}(a\cdot (D+q\widehat{D})b,c)+\mathfrak{B}(b, a\cdot (D+q\widehat{D})c+c\cdot (D+q\widehat{D})a)\\
&=&\mathfrak{B}(a, (1+q)(D(b)+\widehat{D}(b))\cdot c+qb\cdot (D+\widehat{D})(c))\\
&=& \mathfrak{B}(a, (1+2q)(D(b)+\widehat{D}(b))\cdot c).
\end{eqnarray*}}
Then the conclusion follows directly by Proposition \ref{pro:cond}. 
\end{proof}

\begin{defi}
\begin{enumerate}
  \item  \mcite{HBG} A {\bf (standard) Manin triple of Novikov algebras} is a triple $(B,(A,\circ_{A}),(A^\ast,$ $\circ_{A^\ast}))$ of Novikov algebras for which
\begin{itemize}
    \item as a vector space, $B$ is the direct sum of $A$ and $A^\ast$;
    \item $(A,\circ_{A})$ and $(A^\ast,\circ_{A^\ast})$ are Novikov subalgebras of $B$;
    \item the bilinear form on $B=A\oplus A^\ast$ defined by
\begin{eqnarray}\mlabel{bilinear}
\mathcal{B}(a+f,b+g)=\langle f, b\rangle+\langle g, a\rangle,
\;\;\;a,~~b\in A,~~f,~~g\in A^\ast,
\end{eqnarray}
is invariant.
\end{itemize}
  \item \mcite{Bai} A {\bf double construction of a Frobenius algebra} is a triple $((B, \cdot), (A, \cdot_A), (A^\ast, \cdot_{A^\ast}))$ of commutative associative algebras satisfying the following conditions:
\begin{itemize}
\item as a vector space, $B$ is the direct sum of $A$ and $A^\ast$;
\item $(A, \cdot_A)$ and $(A^\ast, \cdot_{A^\ast})$ are  subalgebras of $(B, \cdot)$;
\item the bilinear form on $B=A\oplus A^\ast$ given by Eq. (\mref{bilinear}) is invariant.
\end{itemize}
\end{enumerate}
\end{defi}

\begin{defi} \mcite{LLB}
Let $(A, \cdot_A, D)$ and $(A^\ast, \cdot_{A^\ast}, Q^\ast)$ be  commutative differential algebras.
 A {\bf double construction of a commutative differential Frobenius algebra} associated with  $(A, \cdot_A, D)$ and $(A^\ast, \cdot_{A^\ast}, Q^\ast)$ is a double construction
   of a commutative Frobenius algebra $((B, \cdot), (A, \cdot_A), (A^\ast, \cdot_{A^\ast}))$
such that $(B, \cdot, D+Q^\ast)$ is a commutative differential
algebra.
  We denote it by $((B, \cdot, D+Q^\ast), (A, \cdot_A, D), (A^\ast, \cdot_{A^\ast}, Q^\ast))$.
\end{defi}

\begin{rmk}\mlabel{equi-doub-const}
Let $((B, \cdot, D+Q^\ast), (A, \cdot_A, D), (A^\ast,
\cdot_{A^\ast}, Q^\ast))$ be a double construction of a
commutative differential Frobenius algebra. By
\cite[Lemma 3.5]{LLB}, we have $\widehat{D+Q^\ast}=Q+D^\ast$, and
$(A, \cdot_A, D, Q)$ and $(A^\ast, \cdot_{A^\ast}, Q^\ast,
D^\ast)$ are admissible commutative differential algebras. As a
consequence, the triple $((B, \cdot, D+Q^\ast), (A, \cdot_A, D),
(A^\ast, \cdot_{A^\ast}, Q^\ast))$ is also a double construction
in the context of admissible commutative differential algebras.
\end{rmk}

We give a construction of Manin triples of Novikov algebras from double constructions of commutative differential Frobenius algebras as follows.

\begin{pro}\mlabel{double construction}
Let $(A, \cdot_A, D)$ and $(A^\ast, \cdot_{A^\ast}, Q^\ast)$ be
commutative differential algebras. Suppose that $((B, \cdot,
D+Q^\ast), (A, \cdot_A, D), (A^\ast, \cdot_{A^\ast}, Q^\ast))$ is
a double construction of a commutative differential Frobenius
algebra. Then with the Novikov algebra structure on
$B$ induced from $(B, \cdot, D+Q^\ast, Q+D^\ast)$ and with the
abbreviations $\circ:=\circ_A$ and $\circ^*:=\circ_{A^*}$, the
triple $(B, (A, {\circ}_q), (A^\ast, {\circ^\ast}_{\!\!q}))$ is a
Manin triple of Novikov algebras if and only if $q=-\frac{1}{2}$
or $Q$ is a derivation on $(A, \cdot_A)$ and $D^\ast$ is a
derivation on $(A^\ast, \cdot_{A^\ast})$.
\end{pro}

\begin{proof}
It follows directly from Propositions \mref{cor:constr}.
\end{proof}

We recall the characterizations of Novikov bialgebras and commutative and cocommutative differential ASI bialgebras by Manin triples and double constructions respectively.

\begin{pro}\cite[Corollary 3.13] {HBG}\mlabel{cc1}
Let $(A,\circ_A)$ be a Novikov algebra. Suppose that there is also
a Novikov algebra structure $\circ_{A^\ast}$ on its dual space
$A^\ast$ induced from $\Delta: A\rightarrow A\otimes A$. Then
there is a Manin triple of Novikov algebras associated with
$(A,\circ_A)$ and $(A^\ast, \circ_{A^\ast})$ if and only if  $(A,
\circ_A, \Delta)$ is a Novikov bialgebra.
\end{pro}

\begin{pro}\mlabel{diff-bialg-equ}\cite[Theorem 3.14]{LLB} Let
$(A, \cdot_A, D)$ be a commutative differential algebra. Suppose
that there is a commutative differential algebra structure
$(A^\ast,\cdot_{A^\ast}, Q^\ast)$ on $A^\ast$. Let $\delta:
A\rightarrow A\otimes A$ be the linear dual of the binary
operation $\cdot_{A^\ast}: A^\ast\otimes A^\ast\rightarrow
A^\ast$. Then $(A,\cdot_A, \delta,D,Q)$ is a commutative and
cocommutative differential ASI bialgebra if and only if there is a
double construction of a commutative differential Frobenius
algebra $((B, \cdot, D+Q^\ast), (A, $ $\cdot_A, D),$
$ (A^\ast,$ $ \cdot_{A^\ast}, Q^\ast))$.
\end{pro}

Now we give such a
characterization for the induction of Novikov bialgebras from
commutative and cocommutative differential ASI bialgebras.

\begin{pro}\mlabel{corr-bialgebra}
With the conditions in Theorem \mref{bialgebra-constr}, we
obtain the following conclusions, using the same abbreviations as in Proposition~\mref{double construction}.
\begin{enumerate}
\item The Novikov bialgebra $(A, \circ_{-\frac{1}{2}},
\Delta_{-\frac{1}{2}})$ in Theorem \mref{bialgebra-constr}
\meqref{it:const2} is precisely
 the one obtained from the Manin triple
of Novikov algebras $(B, (A, {\circ}_{-\frac{1}{2}}), (A^\ast,
{\circ^\ast}_{\!\!-\frac{1}{2}}))$ in Proposition \mref{double
construction} when $q=-\frac{1}{2}$. \mlabel{it:corr2}
\item If $Q$ is a derivation on $(A, \cdot_A)$ and $D^\ast$ is a derivation on $(A^\ast, \cdot_{A^\ast})$, then the
Novikov bialgebra $(A, \circ_q, \Delta_q)$ in Theorem
\mref{bialgebra-constr} \meqref{it:const1} is precisely the one obtained from the Manin triple
of Novikov algebras $(B, ( A, {\circ}_q), (A^\ast,
{\circ^\ast}_{\!\!q}))$ in Proposition \mref{double construction}. \mlabel{it:corr1}
    \end{enumerate}
\end{pro}

Proposition~\mref{corr-bialgebra} can be depicted as the following commutative diagram.
\vspb
$$    \xymatrix{
        \txt{\tiny $(A, \cdot_A, \delta, D, Q)$ \\ \tiny a commutative and\\ \tiny cocommutative \\ \tiny differential ASI \\ \tiny bialgebra} \ar@{<->}[rr]^-{\rm Prop. \mref{diff-bialg-equ}} \ar[d]^(.6){\rm Thm. \mref{bialgebra-constr}}&&
         \txt{\tiny$((B, \cdot, D+Q^\ast), (A, \cdot_A, D), (A^\ast, \cdot_{A^\ast}, Q^\ast))$\\ \tiny a double construction of \\ \tiny a commutative differential\\ \tiny Frobenius algebra} \ar[d]_-{\rm Prop. \mref{double construction}}\\
        \txt{\tiny $(A, {\circ}_q, \Delta_q)$\\ \tiny a Novikov bialgebra} \ar@{<->}[rr]^-{\rm Prop. \mref{cc1}} && \txt{\tiny $(B, (A, {\circ}_q), (A^\ast, {\circ^\ast}_{\!\!q}))$ \\ \tiny a Manin triple \\ \tiny of Novikov algebras}
        }
$$

\begin{rmk}
The proof of Proposition~\mref{double construction} (indeed the proof of Proposition~\mref{cor:constr}) clearly shows why the conditions
that either Eqs. (\mref{eq:conda}) and (\mref{eq:condb}) hold or
$q=-\frac{1}{2}$ are required. Then through the equivalences in
the two horizontal implications in the diagram, one gets another
perspective on why these conditions are needed for
Theorem~\mref{bialgebra-constr}. 
\mlabel{rk:square4}
\vspb
\end{rmk}

\begin{proof}[Proof of Proposition~\mref{corr-bialgebra}]
We only prove \meqref{it:corr1}. Item~\meqref{it:corr2} can be proved in the same way.

By Proposition \mref{diff-bialg-equ}, $(A, \cdot_A, \delta, D, Q)$
where $\cdot_A=\cdot$ is a commutative and cocommutative
differential ASI bialgebra if and only if there is a double
construction of a commutative differential Frobenius
algebra $((B, \cdot, D+Q^\ast), (A, \cdot_A, D),
(A^\ast, \cdot_{A^\ast}, Q^\ast))$. Therefore, by
Proposition \mref{double construction}, $(B, (A, {\circ}_{q}),
(A^\ast, {\circ^\ast}_{\!\!q}))$ is a Manin triple of Novikov
algebras. Then the Novikov algebra operation
${\circ^\ast}_{\!\!q}$ on $A^\ast$ is given by
\begin{eqnarray*}
f{\circ^\ast}_{\!\!q} g=f \cdot_{A^\ast} (Q^\ast+qD^\ast)g,\;\;\; f, g\in A^\ast.
\end{eqnarray*}
Therefore, by Proposition \mref{cc1}, there is a Novikov bialgebra
$(A, {\circ}_{q}, \Delta_A)$, where $\Delta_A$ is obtained from
the Novikov algebra $(A^\ast, {\circ^\ast}_{\!\!q})$. Hence for all $a\in A$, $f$, $g\in A^\ast$, we have
\vspb
\begin{eqnarray*}
\langle \Delta_A(a), f\otimes g\rangle &=&\langle a, f{\circ^\ast}_{\!\!q} g\rangle=\langle a, f\cdot_{A^\ast} (Q^\ast+qD^\ast)(g)\rangle\\
&=&\langle \delta(a), f\otimes (Q^\ast+qD^\ast)(g)\rangle=\langle (\id\otimes (Q+qD))\delta(a), f\otimes g\rangle.
\end{eqnarray*}
Therefore, we obtain $\Delta_A=(\id\otimes (Q+qD))\delta=\Delta_q$.
 Then the proof is completed.
 \vspb
\end{proof}

\begin{rmk} \label{rmk:example}
Let $(A,\cdot, \delta,D,Q)$ be a commutative and
cocommutative differential ASI bialgebra. Define a binary
operation $\circ_q$ and a coproduct $\Delta_q$ on $A$ by
Eqs.~\meqref{eq:k} and \meqref{eq:cons2} respectively. Then
$(A,\circ_q)$ is a Novikov algebra and $(A,\Delta_q)$ is a Novikov coalgebra. Furthermore, there are three possibilities:
\begin{enumerate}
\item \label{it:a1}
Suppose that the condition that either
 $q=-\frac{1}{2}$ or $Q$ is a derivation
on $(A, \cdot_A)$ and $D^\ast$ is a derivation on $(A^\ast,
\cdot_{A^\ast})$ is satisfied. Then $(A, \circ_q, \Delta_q)$ is a Novikov bialgebra, and the Novikov algebra structure on $A\oplus A^*$ is
induced from the admissible commutative differential algebra $(A\oplus A^\ast, \cdot, D+Q^\ast, Q+D^\ast)$. Equivalently,
the corresponding Manin triples of Novikov algebras $(B, (A,
\circ_q), (A^\ast, \circ_q^\ast))$ is induced from the
corresponding double construction of a commutative differential
Frobenius algebra $((A\oplus A^\ast, \cdot, D+Q^\ast), (A, \cdot_A, D),
(A^\ast, \cdot_{A^\ast}, Q^\ast))$  by
Proposition~\ref{double construction}.
\item \label{it:a2} Suppose that the condition that either
 $q=-\frac{1}{2}$ or $Q$ is a derivation
on $(A, \cdot_A)$ and $D^\ast$ is a derivation on $(A^\ast,
\cdot_{A^\ast})$ is not satisfied, but the conditions in
Proposition~\ref{prostr1} are satisfied. Then $(A, \circ_q, \Delta_q)$ is still a Novikov bialgebra. However, the Novikov
algebra structure on $A\oplus A^*$ is not induced from the
admissible commutative differential algebra $(A\oplus A^*, \cdot, D+Q^\ast,
Q+D^\ast)$. Equivalently, the corresponding Manin triples of
Novikov algebras $(A\oplus A^*, (A, \circ_q), (A^\ast, \circ_q^\ast))$ is
not induced from the corresponding double construction of a
commutative differential Frobenius algebra $((A\oplus A^*, \cdot, D+Q^\ast),
(A, \cdot_A, D), (A^\ast, \cdot_{A^\ast}, Q^\ast))$ by Proposition~\ref{double construction}.
\item \label{it:a3}
Suppose that the conditions in Proposition~\ref{prostr1} are not satisfied. Then $(A, \circ_q, \Delta_q)$ is not a Novikov bialgebra.
\end{enumerate}
As we will see, Example~\ref{examp-Novikov-2} provides instances for all the three possibilities. In particular, Case \meqref{it:a2} does occur.
Thus, for the condition that either $q=-\frac{1}{2}$ or $Q$ is a derivation
on $(A, \cdot_A)$ and $D^\ast$ is a derivation on $(A^\ast,
\cdot_{A^\ast})$, it is only a sufficient condition, but not a necessary condition for
$(A,\circ_q,\Delta_q)$ to be a Novikov bialgebra. It becomes a sufficient and necessary condition for $(A,\circ_q,\Delta_q)$ to be a Novikov bialgebra after the extra restriction that the Novikov algebra structure on $A\oplus A^*$ is required to be induced from the admissible commutative
differential algebra $(A\oplus A^\ast, \cdot, D+Q^\ast, Q+D^\ast)$ (see Proposition~\mref{double construction}). 
\vspb
\end{rmk}

\section{Novikov Yang-Baxter equation via admissible associative Yang-Baxter equation}
\mlabel{s:ybe} Let $(A, \cdot, D, Q)$ be an admissible commutative
differential algebra. In this section, we will give a construction
of antisymmetric solutions of the Novikov Yang-Baxter
equation in $(A, \circ_q)$ from antisymmetric
solutions of the admissible associative Yang-Baxter equation in
$(A, \cdot, D, Q)$. The relations among $\mathcal{O}$-operators,
admissible differential Zinbiel algebras and pre-Novikov algebras
are established. Moreover, we show that there is a natural
construction of Novikov bialgebras from admissible differential
Zinbiel algebras.

\subsection{NYBE via AYBE: the general case}
Let $(V, \ast)$ be a vector space with a binary operation $\ast$ and take $r=\sum_i x_i\otimes y_i\in V\otimes V$.
Define
\vspb
$$r_{13}\ast r_{23}:=\sum_{i, j}x_i\otimes x_j\otimes y_i \ast y_j,
~~r_{12}\ast r_{23}:=\sum_{i, j} x_i \otimes y_i\ast x_j\otimes y_j,
r_{13}\ast r_{12}:=\sum_{i, j}x_i\ast x_j \otimes y_j \otimes y_i.
$$

Let $(A, \circ)$ be a Novikov algebra and denote $a\star b:=a\circ b+b\circ a$ for $a, b\in A$. The equation
\begin{eqnarray}
r\bullet r\coloneqq r_{13}\circ r_{23} +r_{12}\star r_{23}+r_{13}\circ r_{12}=0
\end{eqnarray}
is called the {\bf Novikov Yang-Baxter equation (NYBE)} in $(A, \circ)$.

\begin{pro}\cite[Corollary 3.23]{HBG} \mlabel{Nov-coboundary} If $r\in A\otimes A$ is an antisymmetric solution of the NYBE in $(A,\circ)$, then $(A, \circ, \Delta_r)$ is a Novikov bialgebra with $\Delta_r$ defined by
\begin{eqnarray}\mlabel{co1}
\Delta_r(a)\coloneqq (L_{\circ}(a)\otimes \id+\id\otimes L_{\star}(a))r,\;\;\;\; a\in A.
\end{eqnarray}
\end{pro}

Let $(A,\cdot)$  be a commutative associative algebra. The equation
\begin{equation}\mlabel{eq:AYBE}
r_{13}\cdot r_{12}+r_{13}\cdot r_{23}-r_{12}\cdot r_{23}=0
\end{equation}
is called the {\bf associative Yang-Baxter equation (AYBE)} in $(A, \cdot)$.

Recall \mcite{Bai} that if $r\in A\otimes A$ is an antisymmetric solution of the AYBE in $(A,\cdot)$, then
there is a commutative and cocommutative ASI bialgebra $(A, \cdot,
\delta_r)$, where $\delta_r$ is defined by
\begin{eqnarray}\mlabel{cobass1}
\delta_r(a)=(\id\otimes L_{\cdot}(a)-L_{\cdot}(a)\otimes \id)r,\;\;\;\;a\in A.
\end{eqnarray}

\begin{defi}  Let $(A,\cdot, D, Q)$ be an admissible commutative
differential algebra and $r\in A\otimes A$.
If $r$ is a solution of the AYBE in $(A, \cdot)$ which
satisfies
\vspb
\begin{eqnarray}
&&(D\otimes \id-\id\otimes Q)(r)=0,\\
&&(\id\otimes D-Q\otimes \id)(r)=0,
\end{eqnarray}
then  $r$ is called a solution of the {\bf admissible associative
Yang-Baxter equation (admissible AYBE)} in $(A,\cdot, D, Q)$.
\end{defi}

\begin{lem}\cite[Corollary 4.9]{LLB} \mlabel{diff-coboundary}
Let $(A, \cdot, D, Q)$ be an admissible commutative differential
algebra and $r\in A\otimes A$. If $r$ is an
antisymmetric solution of the admissible AYBE in $(A, \cdot, D,
Q)$, then $(A, \cdot, \delta_r, D, Q)$ is a commutative and
cocommutative differential ASI bialgebra, where $\delta_r$ is
given by Eq. \meqref{cobass1}.
\end{lem}

Next, we give a construction of solutions of NYBE from
antisymmetric solutions  of admissible AYBE.

\begin{pro}\mlabel{Corr-Yang-Baxter}
Let $(A,\cdot, D, Q)$ be an admissible commutative differential
algebra and $(A, \circ_q)$ be the induced Novikov algebra. Let
$r\in A\otimes A$ be an antisymmetric
solution of the admissible AYBE in $(A,\cdot, D, Q)$.  Then the
following conclusions hold.
\begin{enumerate}
\item \label{it:aa} Unconditionally, $r$ is a solution of the NYBE
in $(A,\circ_{-\frac{1}{2}})$. \item\label{it:bb}  If $Q$ is
moreover a derivation on $(A,\cdot)$, then $r$ is a solution of
the NYBE in $(A,\circ_q)$. In particular, if $r$ is an antisymmetric solution of the admissible
AYBE in $(A,\cdot, D, -D)$, then $r$ is a solution of the NYBE in
$(A,\circ_q)$.
\end{enumerate}
\end{pro}

\begin{proof}
Let $r=\sum_i x_i\otimes y_i$. Then we have
{\wuhao\begin{eqnarray*}
r\bullet r&=&r_{13}\circ_{q} r_{23}+r_{12}\star_q r_{23}+r_{13}\circ_q r_{12}\\
&=& \sum_{i,j}(x_i\otimes x_j\otimes y_i\cdot (D+qQ)y_j+x_i\otimes (y_i\cdot (D+qQ)x_j+x_j\cdot (D+qQ)y_i)\otimes y_j\\
&&\quad+x_i\cdot (D+qQ)x_j\otimes y_j\otimes y_i)\\
&=& \sum_{i,j}(x_i\otimes (Q+qD)x_j\otimes y_i\cdot y_j+x_i\otimes (y_i\cdot (D+qQ)x_j+x_j\cdot (D+qQ)y_i) \otimes y_j\\
&&\quad+x_i\cdot x_j\otimes (Q+qD) y_j\otimes y_i)\\
&=&(\id\otimes (Q+qD)\otimes \id)(r_{12}\cdot r_{13}+r_{13}\cdot r_{23})+\sum_{i,j}(x_i\otimes (y_i\cdot (D+qQ)x_j+x_j\cdot (D+qQ)y_i) \otimes y_j)\\
&=& (\id\otimes (Q+qD)\otimes \id) (r_{23}\cdot r_{12})+\sum_{i,j}(x_i\otimes (y_i\cdot (D+qQ)x_j+x_j\cdot (D+qQ)y_i) \otimes y_j)\\
&=&\sum_{i,j}x_i\otimes ((Q+qD)(x_j\cdot y_i)+y_i\cdot (D+qQ)x_j+x_j\cdot (D+qQ)y_i)\otimes y_j\\
&=& \sum_{i,j}x_i\otimes ((1+2q)(D+Q)(x_j)\cdot y_i)\otimes y_j.
\end{eqnarray*}}
Note that by Proposition \ref{pro:cond}, if $Q$ is a derivation on $(A,\cdot)$, then $(D+Q)(a)\cdot b=0$ for all $a$, $b\in A$.
Thus if $q=-\frac{1}{2}$ or $Q$ is a derivation on $(A,\cdot)$, then $(1+2q)(D+Q)(x_j)\cdot y_i=0$. Then $r$ is a solution of the NYBE.
\end{proof}

Let $(A, \cdot, D, Q)$ be an admissible commutative differential
algebra and $(A, \circ_q)$ be the induced Novikov algebra. Let
$r\in A\otimes A$ be an antisymmetric
solution of the admissible AYBE in $(A,\cdot, D, Q)$. Let
$\delta_r: A\rightarrow A\otimes A$ be the linear map defined by
Eq. \meqref{cobass1}. Then by Lemma \mref{diff-coboundary}, $(A,
\cdot, \delta_r, D, Q)$ is a commutative and cocommutative
differential ASI bialgebra. Note that $(A, \cdot, -\delta_r, D,
Q)$ is also a commutative and cocommutative differential ASI
bialgebra. Let $\Delta_q: A\rightarrow A\otimes A$ be the linear
map obtained from $-\delta_r$ by Eq. \meqref{eq:cons2}. Then we
have the following conclusions:
\begin{enumerate}
\item  On the one hand, by Proposition~\ref{Corr-Yang-Baxter}
(\ref{it:aa}), $r$ is a solution of the NYBE in
$(A,\circ_{-\frac{1}{2}})$. Hence by
Proposition~\ref{Nov-coboundary}, there is a Novikov bialgebra
$(A, \circ_{-\frac{1}{2}}, \Delta_{-\frac{1}{2}, r})$, where
$\Delta_{-\frac{1}{2}, r}$ is defined by Eq.~\meqref{co1}. On the
other hand, by Theorem \mref{bialgebra-constr} (\ref{it:aaa}),
there is a Novikov bialgebra $(A, \circ_{-\frac{1}{2}},
\Delta_{-\frac{1}{2}})$.
\item \mlabel{it:cob1} Suppose that $Q$ is a derivation on $(A,\cdot)$.  On the one hand, by
Proposition~\ref{Corr-Yang-Baxter} (\ref{it:bb}), $r$ is a
solution of the NYBE in $(A,\circ_q)$. Hence by
Proposition~\ref{Nov-coboundary}, there is a Novikov bialgebra
$(A, \circ_q, \Delta_{q, r})$, where $\Delta_{q, r}$ is defined by
Eq.~\meqref{co1}. On the other hand, note that in this case, Eqs.
(\mref{eq:conda}) and (\ref{eq:condb}) hold. Hence by Theorem \mref{bialgebra-constr}
(\ref{it:bbb}), there is a Novikov bialgebra $(A, \circ_q,
\Delta_{q})$.
\end{enumerate}

\begin{pro}\mlabel{pp:rmatbialg} With above notations, we
have the following conclusions.
\begin{enumerate}
\item \mlabel{it:cob2} The two Novikov bialgebras $(A,
\circ_{-\frac{1}{2}}, \Delta_{-\frac{1}{2}, r})$ and  $(A,
\circ_{-\frac{1}{2}}, \Delta_{-\frac{1}{2}})$ coincide.
\item
Suppose that  $Q$ is a derivation on $(A,\cdot)$. Then for each $q\in {\bf
k}$, the two Novikov bialgebras $(A, \circ_q, \Delta_{q, r})$ and $(A,
\circ_q, \Delta_{q})$ coincide.
\end{enumerate}
That is, the following diagram commutes.
\vspb
$$  \xymatrix{
    \txt{\tiny $r$ \\ \tiny an antisymmetric solution\\ \tiny of the admissible AYBE in $(A, \cdot, D, Q)$} \ar[rr]^-{\rm Lem. \mref{diff-coboundary}}   \ar[d]_-{\rm Prop. \mref{Corr-Yang-Baxter}}&& \txt{ \tiny $(A,\cdot,\delta_r)$ \\ \tiny a commutative and cocommutative\\ \tiny differential ASI bialgebra} \ar[d]_-{\rm Thm. \mref{bialgebra-constr}}\\
    \txt{\tiny $r$\\ \tiny an antisymmetric solution \\
    	\tiny of the NYBE in $(A,\circ_q)$} \ar[rr]^-{\rm
        Prop. \mref{Nov-coboundary}}
    && \txt{\tiny $(A,\circ_q,\Delta_{q,r})=(A,\circ_q,\Delta_q)$ \\ \tiny a Novikov bialgebra}  }
\vspb
$$
\end{pro}

\begin{proof}
We only prove \meqref{it:cob1}. Item~\meqref{it:cob2} can be similarly proved.

Let $r=\sum_i x_i\otimes y_i$ and $a\in A$. Then we have
{\wuhao\begin{eqnarray*}
\Delta_q(a)&=& (\id \otimes (Q+qD))(-\delta_r)(a)\\
&=&-(\id\otimes (Q+qD))\sum_i(x_i\otimes a\cdot y_i-a\cdot x_i\otimes y_i)\\
&=&-\sum_i(x_i\otimes (Q+qD)(a\cdot y_i)-a\cdot x_i\otimes (Q+qD)(y_i))\\
&=& -\sum_i (x_i\otimes (Q(a)\cdot y_i-a\cdot D(y_i)+qD(a)\cdot y_i+qa\cdot D(y_i))-a\cdot (D+qQ)x_i\otimes y_i)\\
&=&\sum_i(x_i\otimes (a\cdot (D+qQ)y_i+y_i\cdot (D+qQ)a-(1+2q)(D+Q)(a)\cdot y_i)+a\cdot (D+qQ)(x_i)\otimes y_i)\\
&=&\sum_i (a\circ_q x_i\otimes y_i+x_i\otimes (a\circ_q y_i+y_i\circ_q a))\\
&=& (L_{\circ_q}(a)\otimes \id+\id\otimes
L_{\star_q}(a))r=\Delta_{q,r}.
\end{eqnarray*}}
The proof is completed.
\vspb
\end{proof}

\subsection{NYBE via AYBE: $\mathcal O$-operators}
Recall the notion of representations of Novikov algebras.
\begin{defi} \mcite{O}
A {\bf representation} of a Novikov algebra $(A,\circ)$ is a triple $(V, l_A,r_A)$, where $V$ is a vector space and   $l_A$, $r_A: A\rightarrow {\rm
End}_{\bf k}(V)$ are linear maps satisfying
\vspb
\begin{eqnarray}
\mlabel{lef-mod1}&l_A(a\circ b-b\circ a)v=l_A(a)l_A(b)v-l_A(b)l_A(a)v,&\\
\mlabel{lef-mod2}&l_A(a)r_A(b)v-r_A(b)l_A(a)v=r_A(a\circ b)v-r_A(b)r_A(a)v,&\\
\mlabel{Nov-mod1}&l_A(a\circ b)v=r_A(b)l_A(a)v,&\\
\mlabel{Nov-mod2}&r_A(a)r_A(b)v=r_A(b)r_A(a)v,\quad a, b\in A, v\in V.&
\end{eqnarray}
\end{defi}

Note that $(A, L_{\circ}, R_{\circ})$ is a representation of $(A,\circ)$, called the
\textbf{adjoint representation} of $(A,\circ)$.

\begin{pro}\cite[Proposition 3.2]{HBG}\mlabel{pro:semi}
Let $(A,\circ)$ be a Novikov algebra. Let $V$ be a vector space
and $l_A,r_A: A\to {\rm End}_{\bf k}(V)$ be linear maps.
Define a binary operation $\circ$ 
on the direct sum $A\oplus V$ of the underlying vector spaces of $A$ and $V$ by
\vspb
    $$(a+u)\circ (b+v)\coloneqq a\circ b+l_A(a)v+r_A(b)u, \quad\quad a, b\in A,\;u, v\in V.$$
Then $(V,l_A,r_A)$ is a representation of $(A,\circ)$ if and only
if $(A\oplus V,\circ)$ is a Novikov algebra, called the
{\bf semi-direct product} of $(A, \circ)$ by its representation $(V,l_A,r_A)$ and denoted by
$A\ltimes_{l_A,r_A} V$.
\end{pro}

\begin{defi} \mcite{LLB}\label{Def:mo}
A {\bf representation} of a commutative differential algebra $(A, \cdot, D)$ is
a triple $(V, l_A, \alpha)$ where $(V, l_A)$ is a representation of  $(A, \cdot)$ and $\alpha: V\rightarrow V$ is a linear map satisfying
\begin{eqnarray}\mlabel{diff-module}
\alpha (l_A(a)v)=l_A(D(a))v+l_A(a)\alpha(v),\;\;\;\; a\in A, v\in V.
\end{eqnarray}
\end{defi}

Next, we introduce the notion of a representation of an admissible
commutative differential algebra.

\begin{defi}
Let $(A, \cdot, D, Q)$ be an admissible commutative differential algebra and $(V, l_A, \alpha)$ be a representation of $(A, \cdot, D)$. If a linear map $\beta: V\rightarrow V$ satisfies
\begin{eqnarray*}
\beta(l_A(a)v)=l_A(a)\beta(v)-l_A(D(a))v=l_A(Q(a))v-l_A(a)\alpha(v),\;\;a\in A, v\in V,
\end{eqnarray*}
then we say that $(V, l_A, \alpha, \beta)$ is a {\bf representation } of $(A, \cdot, D, Q)$.
\end{defi}

\begin{pro}\mlabel{pro:semi-ACDA}
Let $(A,\cdot, D, Q)$ be an admissible commutative differential algebra. Let $V$ be a vector space
, $l_A: A\to {\rm End}_{\bf k}(V)$ be a linear map and $\alpha$, $\beta: V\rightarrow V$ be linear maps.
Define a binary operation $\cdot$ on the direct sum $A\oplus V$ of the underlying vector spaces of $A$ and $V$ by
    $$(a+u)\cdot (b+v)\coloneqq a\cdot b+l_A(a)v+l_A(b)u, \quad\quad a, b\in A,\;u, v\in V.$$
Then $(V,l_A,\alpha, \beta)$ is a representation of $(A,\cdot, D, Q)$ if and only
if $(A\oplus V,\cdot, D+\alpha, Q+\beta)$ is an admissible commutive differential algebra, called the
{\bf semi-direct product} of $(A,\cdot, D, Q)$ by its representation $(V,l_A,\alpha, \beta)$ and denoted by
$(A\ltimes_{l_A} V, \cdot, D+\alpha, Q+\beta)$.
\end{pro}
\begin{proof}
It is straightforward.
\end{proof}

\begin{pro}\mlabel{adm-reprent}
Let $(A, \cdot, D, Q)$ be an admissible commutative differential
algebra, $(V, l_A, \alpha, \beta)$ be a representation of $(A,
\cdot, D, Q)$ and $q\in {\bf k}$. Let $(A, \circ_q)$ be the
Novikov algebra induced from $(A, \cdot, D, Q)$. Define linear
maps $l_{l_A,q}$, $r_{l_A,q}: A\mapsto
\End_{\bf k}(V)$ by 
\begin{eqnarray}
l_{l_A,q}(a)v:=l_A(a)(\alpha+q\beta)v,\;\;
r_{l_A,q}(a)v:=l_A((D+qQ)a)v,\;\;a\in A, v\in V.
\end{eqnarray}
Then $(V, l_{l_A,q}, r_{l_A,q})$ is a representation of $(A, \circ_q)$.
\end{pro}

\begin{proof}
Let $a$, $b\in A$, $v\in V$ and $u=(\alpha+q\beta)v$. Then we verify Eq.~\meqref{lef-mod1} as follows.
{\small
\begin{eqnarray*}
&&l_{l_A,q}(a\circ_q b-b\circ_q a)v-l_{l_A,q}(a)l_{l_A,q}(b)v+l_{l_A,q}(b)l_{l_A,q}(a)v\\
&&\quad=l_A(a\cdot (D+qQ)b-b\cdot (D+qQ)a)((\alpha+q\beta)v)-l_A(a)((\alpha+q\beta)(l_A(b)(\alpha+q\beta)v))\\
&&\qquad+l_A(b)((\alpha+q\beta)(l_A(a)(\alpha+q\beta)v))\\
&&\quad=l_A(a\cdot (D+qQ)b-b\cdot (D+qQ)a)u-l_A(a)(\alpha(l_A(b)u))-ql_A(a)\beta(l_A(b)u)\\
&&\qquad+l_A(b)\alpha(l_A(a)u)+ql_A(b)\beta(l_A(a)u)\\
&&\quad=l_A(a\cdot (D+qQ)b-b\cdot (D+qQ)a)u-ql_A(a)(l_A(b)\beta(u)-l_A(D(b))u)-l_A(a)(l_A(b)\alpha(u)+l_A(D(b))u)\\
&&\qquad+ql_A(b)(l_A(a)\beta(u)-l_A(D(a))u)+l_A(b)(l_A(D(a))u+l_A(a)\alpha(u))\\
&&\quad=l_A(a\cdot D(b))u+ql_A(a\cdot Q(b))u-l_A(b\cdot D(a))u-ql_A(b\cdot Q(a))u-l_A(a)l_A(D(b))u\\
&&\qquad-l_A(a)l_A(b)\alpha(u)-ql_A(a)l_A(b)\beta(u)+ql_A(a)l_A(D(b))u+l_A(b)l_A(D(a))u+l_A(b)l_A(a)\alpha(u)\\
&&\qquad+ql_A(b)l_A(a)\beta(u)-ql_A(b)l_A(D(a))u\\
&&\quad=ql_A(a\cdot Q(b)-b\cdot Q(a)+a\cdot D(b)-b\cdot D(a))u\\
&&\quad= 0.
\end{eqnarray*}
}
Eqs.~\meqref{lef-mod2}--\meqref{Nov-mod2} can be similarly checked, completing the proof.
\end{proof}

\begin{pro}\label{pro:semi-direct}
Let $(A, \cdot, D, Q)$ be an admissible commutative differential
algebra, $(V, l_A, \alpha, \beta)$ be a representation of $(A,
\cdot, D, Q)$ and $q\in {\bf k}$. Then the Novikov algebra induced
from the admissible commutative differential algebra
$(A\ltimes_{l_A} V, \cdot, D+\alpha, Q+\beta)$ is exactly the semi-direct product
$A\ltimes_{l_{l_A, q},r_{l_A,q}}V$ of $(A,
\circ_q)$ by $(V, l_{l_A,q}, r_{l_A,q})$.
\end{pro}
\begin{proof}
It is straightforward.
\end{proof}
For a linear map $\varphi:A\rightarrow\mathrm{End}_{
\bf k}(V)$, define a linear map
$\varphi^{*}:A\rightarrow\mathrm{End}_{\bf k}(V^{*})$ by
$$\langle\varphi^{*}(a)f,v\rangle=-\langle f,\varphi(a)v\rangle,\;\;\; a\in A, f\in V^{*},v\in V,$$
where $\langle\cdot, \cdot\rangle$ is the usual pairing between $V$ and $V^*$.

\begin{pro}\mlabel{ad-dual-rep}
Let $(A, \cdot, D, Q)$ be an admissible commutative differential algebra and $(V, l_A, \alpha, \beta)$ be a representation of $(A, \cdot, D, Q)$. Then $(V^\ast, -l_A^\ast, \beta^\ast, \alpha^\ast)$ is a representation of $(A, \cdot, D, Q)$.
\end{pro}

\begin{proof}
By \cite[Lemma 2.7]{LLB}, $(V^\ast, -l_A^\ast, \beta^\ast)$ is a representation of $(A, \cdot, D)$.
Let $a\in A$, $f\in V^\ast$ and $v\in V$. Then we have
\begin{eqnarray*}
&&\langle \alpha^\ast((-l_A^\ast)(a)f)+l_A^\ast(a)\alpha^\ast(f)-l_A^\ast(D(a))f, v\rangle
=\langle f, l_A(a) \alpha(v)-\alpha(l_A(a)v)+l_A(D(a))v\rangle=0,\\
&&\langle \alpha^\ast((-l_A^\ast)(a)f)+l_A^\ast(Q(a))f-l_A^\ast(a)\beta^\ast(f), v\rangle=\langle f, l_A(a) \alpha(v)-l_A(Q(a))v+\beta(l_A(a)v)\rangle=0.
\end{eqnarray*}
Hence the conclusion follows.
\end{proof}

Note that $(A, L_\cdot, D, Q)$ is a representation of $(A, \cdot, D, Q)$. Then we have the following conclusion.
\begin{cor}\mlabel{admi-rep}
If $(A, \cdot, D, Q)$ is an admissible commutative differential
algebra, then $(A^\ast, -L_{\cdot}^\ast, Q^\ast, D^\ast)$ is a representation of $(A, \cdot, D, Q)$.
\end{cor}

\begin{pro} \cite[Proposition 3.3]{HBG}\mlabel{pp:dualrep}
Let $(A,\circ)$ be a Novikov algebra and $(V, l_A, r_A)$ be a
representation of $(A,\circ)$. Then $(V^\ast, l_A^\ast+r_A^\ast, -r_A^\ast)$ is
a representation of $(A,\circ)$.
\end{pro}
\begin{rmk}
Since $(A, L_\circ, R_\circ)$ is a representation of $(A, \circ)$, $(A^\ast, L_\circ^\ast+R_\circ^\ast, -R_\circ^\ast)=(A^\ast, L_\star^\ast, -R_\circ^\ast)$ is a representation of $(A,\circ)$.
\end{rmk}

We recall the notion of an $\mathcal O$-operator on a Novikov algebra associated to a representation.

\begin{defi}\mcite{HBG}
Let $(A,\circ)$ be a Novikov algebra and $(V, l_A, r_A)$ be a
representation. A linear map $T: V\rightarrow A$ is called an {\bf
$\mathcal{O}$-operator} (also called a {\bf relative Rota-Baxter operator}) on $(A,\circ)$ associated to $(V, l_A,
r_A)$, if $T$ satisfies
\begin{eqnarray*}
T(u)\circ T(v)=T(l_A(T(u))v)+T(r_A(T(v))u),\;\;\;\;u,~~v\in
V.
\end{eqnarray*}
\end{defi}

Let $A$ and $V$ be vector spaces. Then any element  $r\in
A\ot A$ is identified as a linear map $T^r: A^*\rightarrow A$
under the linear bijection $A\otimes A \cong \text{Hom}_{{\bf
k}}(A^{\ast}, A)$ and any linear map $T:V\rightarrow A$ is
identified as $r_T\in A\otimes V^\ast \subseteq (A\oplus
V^*)\otimes (A\oplus V^*)$ through the linear bijection
 ${\rm Hom}_{\bf k}(V, A)\cong A\otimes V^\ast$.
With this identification, the relationship between antisymmetric solutions of NYBE and $\mathcal{O}$-operators
is given as follows.

\begin{pro}\mlabel{Nov-o-operator}
\begin{enumerate}
  \item \cite[Theorem 3.27]{HBG}
Let $(A, \circ)$ be a Novikov algebra and $r\in A\otimes A$ be
antisymmetric.  Then $r$ is a solution of the
NYBE in $(A,\circ)$ if and only if $T^r$ is an $\mathcal
O$-operator on $(A,\circ)$ associated to the representation
$(A^*,L_{\star}^*,-R_{\circ}^*)$. \item \cite[Theorem 3.29]{HBG}
 Let $(A, \circ)$ be a Novikov algebra, $(V, l_A, r_A)$ be a representation of $(A, \circ)$,
and $T: V\rightarrow A$ be a linear map. Then $r=r_T-\tau r_T$ is
a solution of the NYBE in the semi-direct product Novikov algebra
$A\ltimes_{l_A^\ast+r_A^\ast,-r_A^\ast} V^\ast$ if and only if $T$
is an $\mathcal{O}$-operator on $(A,\circ)$ associated to $(V,
l_A, r_A)$.
\end{enumerate}
 \end{pro}

We also recall the notion of
$\calo$-operators on commutative differential algebras.

\begin{defi}\mcite{LLB}
Let $(A, \cdot, D)$ be a commutative differential algebra and $(V, l_A, \alpha)$ be a representation of $(A, \cdot, D)$. A linear map $T: V\rightarrow A$ is called an {\bf $\mathcal{O}$-operator on $(A, \cdot, D)$ associated to $(V, l_A, \alpha)$}, if
\begin{eqnarray}
&&DT=T\alpha,\\
&&T(u)\cdot T(v)=T(l_A(T(u))v+l_A(T(v))u),\;\;\;\; u, v\in V.
\end{eqnarray}
\end{defi}

Now we introduce the notion of an $\mathcal{O}$-operator on
an admissible commutative differential algebra associated to a
representation.

\begin{defi}
Let $(A, \cdot, D, Q)$ be an admissible commutative differential algebra and $(V, l_A, \alpha, \beta)$ be a representation of $(A,\cdot, D, Q)$. A linear map $T: V\rightarrow A$ is called an {\bf $\mathcal{O}$-operator on $(A, \cdot, D, Q)$ associated to $(V, l_A, \alpha, \beta)$}, if $T$ is an $\mathcal{O}$-operator on $(A, \cdot, D)$ associated to $(V, l_A, \alpha)$ and
\begin{eqnarray}
&& QT=T\beta.
\end{eqnarray}
\end{defi}

\begin{pro}\mlabel{Ass-operator1}  \cite[Corollary 4.15]{LLB}
Let $(A, \cdot, D, Q)$ be an admissible commutative differential
algebra and $r\in A\otimes A$ be antisymmetric. Then $r$ is a
solution of the admissible AYBE in $(A, \cdot, D, Q)$ if and only
if $T^r$ is an $ \mathcal{O}$-operator on $(A, \cdot, D)$
associated to the representation $(A^\ast, -L_{\cdot}^\ast,
Q^\ast)$.
\end{pro}

In fact, Proposition~\ref{Ass-operator1} can be restated in terms of $\mathcal O$-operators on admissible commutative
differential algebras as follows.

\begin{cor}\mlabel{s-O-oper}
Let $(A, \cdot, D, Q)$ be an admissible commutative differential
algebra and $r\in A\otimes A$ be antisymmetric. Then $r$ is a
solution of the admissible AYBE in $(A, \cdot, D, Q)$ if and only
if $T^r$ is an $\mathcal{O}$-operator on $(A, \cdot, D, Q)$ associated to $(A^\ast, -L_{\cdot}^\ast, Q^\ast, D^\ast)$.
\end{cor}

\begin{proof}
Suppose that $T^r$ is an $\mathcal{O}$-operator on $(A, \cdot, D, Q)$ associated to $(A^\ast, -L_{\cdot}^\ast, Q^\ast, D^\ast)$. Then $T^r$ is an $\mathcal{O}$-operator on $(A, \cdot, D)$ associated to $(A^\ast, -L_{\cdot}^\ast, Q^\ast)$.
By Proposition~\ref{Ass-operator1},  $r$ is a
solution of the admissible AYBE in $(A, \cdot, D, Q)$.

Suppose that  $r$ is a
solution of the admissible AYBE in $(A, \cdot, D, Q)$.
By Corollary \mref{admi-rep}, $(A^\ast, -L_{\cdot}^\ast, Q^\ast,
D^\ast)$ is a representation of $(A, \cdot, D, Q)$.  By
Proposition \mref{Ass-operator1}, $T^r$ is an $
\mathcal{O}$-operator on $(A, \cdot, D)$ associated to the representation $(A^\ast,
-L_{\cdot}^\ast, Q^\ast)$. Therefore, we only need to check
$QT^r=T^rD^\ast$.
Let $f,g\in A^\ast$. Since $\langle g, T^r(f)\rangle=\langle
g\otimes f, r\rangle$, we have
\begin{eqnarray*}
\langle g, {T^r}^\ast(f)\rangle=\langle f, T^r(g)\rangle=\langle f\otimes g , r\rangle=-\langle g\otimes f , r\rangle=-\langle g, T^r(f)\rangle.
\end{eqnarray*}
This gives ${T^r}^\ast=-T^r$. Since
$DT^r=QT^r$, we have $$T^rD^\ast=-{T^r}^\ast
D^\ast=-(DT^r)^\ast=-(QT^r)^\ast=-{T^r}^\ast Q^\ast=T^rQ^\ast,$$
as required. Therefore $T^r$ is an $\mathcal{O}$-operator on $(A, \cdot, D, Q)$ associated to $(A^\ast, -L_{\cdot}^\ast, Q^\ast, D^\ast)$.
\end{proof}

Next, we construct $\mathcal{O}$-operators on Novikov algebras from $\mathcal{O}$-operators on  admissible commutative differential algebras.

\begin{pro}\mlabel{corr-O-operator-2}
Let $(A, \cdot, D, Q)$ be an admissible commutative differential
algebra and $(A, \circ_{q})$ be the Novikov algebra induced from
$(A, \cdot, D, Q)$. Let $T: V\rightarrow A$ be an
$\mathcal{O}$-operator on $(A, \cdot, D, Q)$ associated to a
representation $(V, l_A, \alpha, \beta)$. Then $T$ is an
$\mathcal{O}$-operator on $(A, \circ_{q})$ associated to $(V,
l_{l_{A},q}, r_{l_{A},q})$. 
\end{pro}

\begin{proof}
For $u$, $v\in V$, we have
{\wuhao \begin{eqnarray*}
T(u)\circ_{q} T(v)&=& T(u)\cdot (D+qQ)T(v)=T(u)\cdot T((\alpha+q\beta)v)\\
&=&T(l_A(T(u))(\alpha+q\beta)v+l_A(T((\alpha+q\beta)v))u)=T(l_{l_A,q}(T(u))v+l_A((D+qQ)T(v))u)\\
&=&T(l_{l_A,q}(T(u))v+r_{l_A,q}(T(v))u).
\end{eqnarray*}}
Thus $T$ is an $\mathcal{O}$-operator on $(A, \circ_{q})$ associated to $(V, l_{l_A,q}, r_{l_A,q})$.
\end{proof}

Let $(A, \cdot, D, Q)$ be an admissible commutative differential
algebra and $(A, \circ_{q})$ be the Novikov algebra induced from
$(A, \cdot, D, Q)$. Let $r$ be an
antisymmetric solution of the admissible AYBE in $(A, \cdot, D,
Q)$. On the one hand, by Corollary \mref{s-O-oper}, $T^r$ is an
$\mathcal{O}$-operator on $(A, \cdot, D, Q)$ associated to
$(A^\ast, -L_{\cdot}^\ast, Q^\ast, D^\ast)$. Thus by Proposition
\mref{corr-O-operator-2}, $T^r$ is an $\mathcal{O}$-operator on
$(A, \circ_q)$ associated to $(A^\ast, l_{-L_\cdot^\ast,q},
r_{-L_\cdot^\ast,q})$. On the other hand, by Proposition
\mref{Corr-Yang-Baxter}, if $q=-\frac{1}{2}$ or $Q$ is a
derivation on $(A,\cdot)$, $r$ is an antisymmetric solution of the
NYBE in $(A,\circ_q)$. Then by Proposition \mref{Nov-o-operator},
$T^r$ is an $\mathcal{O}$-operator on $(A, \circ_q)$ associated to
$(A^\ast, L_{\star_q}^\ast,-R_{\circ_q}^\ast)$.

\begin{pro}
With the notations above, if $q=-\frac{1}{2}$ or $Q$ is a derivation on $(A,\cdot)$, we have
\begin{eqnarray}
l_{-L_\cdot^\ast,q}=L_{\star_q}^\ast,\;\;r_{-L_\cdot^\ast,q}=-R_{\circ_q}^\ast.
\end{eqnarray}
That is, the two approaches to obtain $T^r$ as an $\mathcal{O}$-operator on the Novikov algebra $(A, \circ_q)$ coincide, as shown by the following commutative diagram.
\vspb
\begin{displaymath}
\xymatrix{
\txt{\tiny $r$\\ \tiny an antisymmetric solution of \\ \tiny the admissible AYBE in $(A, \cdot, D, Q)$ } \ar[rr]^-{\rm Cor. \mref{s-O-oper}}   \ar[d]_-{\rm Prop. \mref{Corr-Yang-Baxter}}&& \txt{ \tiny $T^r$\\ \tiny an $\mathcal{O}$-operator on $(A, \cdot, D, Q)$\\ \tiny  associated to $(A^\ast, -L_{\cdot}^\ast, Q^\ast, D^\ast)$} \ar[d]_-{\rm Prop. \mref{corr-O-operator-2}}\\
\txt{\tiny $r$\\ \tiny an antisymmetric solution of \\ \tiny the NYBE in $(A, \circ_{q})$ } \ar[rr]^-{\rm
Prop. \mref{Nov-o-operator}}
            && \txt{\tiny $T^r$\\ \tiny an $\mathcal{O}$-operator on $(A, \circ_{q})$ \\ \tiny associated to $(A^\ast, L_{ \star_q}^\ast, -R_{\circ_{q}}^\ast)=(A^\ast, l_{-L_\cdot^\ast,q}, r_{-L_\cdot^\ast,q})$ }}
\end{displaymath}
\end{pro}

\begin{proof}
For all $a$, $b\in A$ and $f\in A^\ast$, we obtain
\vspb
{\wuhao \begin{eqnarray*}
\langle l_{-L_{\cdot}^\ast,q}(a)f-L_{{\star}_q}^\ast(a)f, b\rangle&=&\langle (Q^\ast+qD^\ast)f, a\cdot b\rangle+\langle f, a{\star}_q b\rangle\\
&=& \langle f, (Q+qD)(a\cdot b)+a\cdot(D+qQ)b+b\cdot(D+qQ)a\rangle\\
&=&\langle f, (1+q)(D+Q)(a)\cdot b+qa\cdot (D+Q)(b)\rangle\\
&=&\langle f,(1+2q)(D+Q)(a)\cdot b\rangle.
\end{eqnarray*}}
Therefore,  if either $q=-\frac{1}{2}$ or $Q$ is a derivation on
        $(A,\cdot)$, the latter amounts to
$D(a)\cdot b=-Q(a)\cdot b$ for all $a, b\in A$ thanks to Proposition \ref{pro:cond},
 we have $\langle l_{-L_{\cdot}^\ast, q}(a)f-L_{{\star}_q}^\ast(a)f, b\rangle=0$ and hence
$l_{-L_{\cdot}^\ast, q}=L_{{\star}_q}^\ast$. Similarly, we have $r_{-L_{\cdot}^\ast,q }=-R_{{\circ}_{q}}^\ast$. Then the conclusion follows directly.
\end{proof}

\begin{pro}\cite[Theorem 4.18]{LLB}\mlabel{O-Oper-2}
Let $(A, \cdot, D, Q)$ be an admissible commutative differential
algebra. Suppose that $(V, l_A, \alpha, \beta)$ is a
representation of $(A, \cdot, D, Q)$ (and hence $(V^\ast, -l_A^\ast,
\beta^\ast, \alpha^\ast)$ is a representation of $(A,
\cdot, D, Q)$). Let $T:V\rightarrow A$ be a linear map. Then
$r=r_T-\tau r_T$ is an antisymmetric solution of the admissible
AYBE in the admissible commutative differential algebra $(
A\ltimes_{-l_A^\ast} V^\ast, \cdot, D+\beta^\ast, Q+\alpha^\ast)$
if and only if $T$ is an $\mathcal{O}$-operator on $(A,\cdot,D,
Q)$ associated to $(V, l_A, \alpha, \beta)$.
\end{pro}

Let $(A, \cdot, D, Q)$ be an admissible commutative differential algebra, $(A, \circ_q)$ be the Novikov algebra induced from $(A, \cdot, D, Q)$, and $T$ be an $\mathcal{O}$-operator on $(A, \cdot, D, Q)$
associated to the representation $(V, l_A, \alpha, \beta)$. On the
one hand, by Proposition \mref{O-Oper-2}, $r=r_T-\tau r_T$ is an antisymmetric solution of the admissible AYBE in
$(A\ltimes_{-l_A^\ast} V^\ast,\cdot, D+\beta^\ast,
Q+\alpha^\ast)$. Then by Proposition \mref{Corr-Yang-Baxter}, if
$q=-\frac{1}{2}$ or $Q+\alpha^\ast$ is a derivation on
$(A\ltimes_{-l_A^\ast} V^\ast,\cdot)$, then $r=r_T-\tau r_T$ is an antisymmetric solution of the NYBE in
$A\ltimes_{l_{-l_A^\ast, q}, r_{-l_A^\ast, q}}V^\ast$. On the
other hand, by Proposition \mref{corr-O-operator-2}, $T$ is an
$\mathcal{O}$-operator on $(A, \circ_{q})$ associated to $(V,
l_{l_A,q}, r_{l_A,q})$. Then by Proposition \mref{Nov-o-operator},
$r=r_T-\tau r_T$ is an antisymmetric
solution of the NYBE in $A\ltimes_{l_{l_A,q}^\ast+r_{l_A,q}^\ast,
-r_{l_A,q}^\ast} V^\ast$.

\begin{pro}\label{pro:o-ybe}
With the notations above, if $q=-\frac{1}{2}$ or
\vspb
\begin{eqnarray}\mlabel{eq:cond2}
l_A(a)(\alpha+\beta)(v)=0,\;\;a\in A, v\in V,
\end{eqnarray}
then
\vspb
$$l_{l_A,q}^\ast+r_{l_A,q}^\ast=l_{-l_A^\ast, q},\;\;-r_{l_A,q}^\ast=r_{-l_A^\ast,q}.$$
Furthermore, if $q=-\frac{1}{2}$ or Eq. \meqref{eq:cond2} and the following extra condition are satisfied
\begin{eqnarray}\mlabel{eq:condd2}
a\cdot (D+Q)(b)=0,\;\;l_A((D+Q)(a))v=0,\;\;(\alpha+\beta)(l_A(a)v)=0,\;\;\;\;a, b\in A,\;\;v\in V,
\end{eqnarray}
then the two approaches that start with an $\calo$-operator $T$ to
obtain $r$ as an antisymmetric solution of
the NYBE coincide, as shown by the commutativity of the following
diagram.

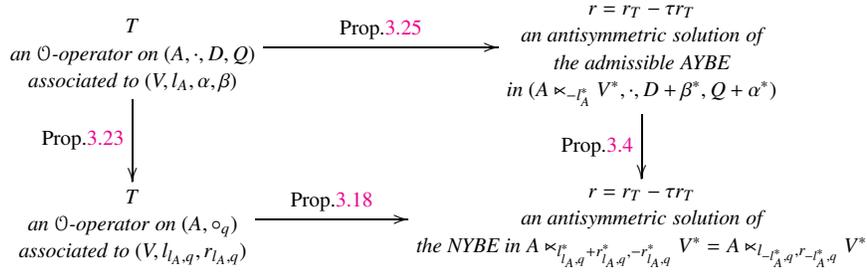
\begin{equation*}
    \begin{split}
\xymatrix{
\txt{\tiny $T$\\ \tiny an $\mathcal{O}$-operator on $(A, \cdot, D, Q)$  \\\tiny associated to $(V, l_A, \alpha, \beta)$ } \ar[rr]^-{\rm Prop. \mref{O-Oper-2}}   \ar[d]_-{\rm Prop. \mref{corr-O-operator-2}}&& \txt{\tiny $r=r_T-\tau r_T$\\ \tiny an antisymmetric solution of \\ \tiny the admissible AYBE \\ \tiny in  $(A\ltimes_{-l_A^\ast} V^\ast,\cdot, D+\beta^\ast, Q+\alpha^\ast)$ } \ar[d]_-{\rm Prop. \mref{Corr-Yang-Baxter}}\\
\txt{\tiny $T$\\  \tiny an  $\mathcal{O}$-operator on $(A, \circ_{q})$\\ \tiny associated to $(V, l_{l_A,q}, r_{l_A,q})$ } \ar[rr]^-{\rm
Prop. \mref{Nov-o-operator}}
            && \txt{\tiny $r=r_T-\tau r_T$\\ \tiny an antisymmetric solution of\\  \tiny the NYBE in $A\ltimes_{l_{l_A,q}^\ast+r_{l_A,q}^\ast, -r_{l_A,q}^\ast} V^\ast=A\ltimes_{l_{-l_A^\ast,q}, r_{-l_A^\ast, q}} V^\ast$ }}
\end{split}
\mlabel{eq:o-ybe}
\end{equation*}
\end{pro}

\begin{proof}
 Let $a\in A$, $f\in V^\ast$ and
$v\in V$. Then we have
\vspb
{\wuhao\begin{eqnarray*}
\langle (l_{l_A,q}^\ast+r_{l_A,q}^\ast)(a)f, v\rangle &=& -\langle f, l_{l_A,q}(a)v+r_{l_A,q}(a)v \rangle\\
&=& -\langle f, l_A(a)(\alpha+q\beta)(v)+l_A((D+qQ)(a))v\rangle,\\
\langle l_{-l_A^\ast,q}(a)f, v\rangle &=& -\langle l_A^\ast(a)(\beta^\ast+q\alpha^\ast)(f), v\rangle\\
&=&\langle (\beta^\ast+q\alpha^\ast)(f), l_A(a)v\rangle=\langle f, (\beta+q\alpha)(l_A(a)v) \rangle\\
&=&\langle f,
l_A(a)\beta(v)-l_A(D(a))v+ql_A(D(a))v+ql_A(a)\alpha(v)\rangle.
\end{eqnarray*}}
Hence we obtain
\vspb
{\wuhao \begin{eqnarray*}
&&\langle l_{-l_A^\ast,q}(a)f-(l_{l_A,q}^\ast+r_{l_A,q}^\ast)(a)f, v\rangle\\
&=&\langle f, (1+q)l_A(a)\beta(v)+(1+q)l_A(a)\beta(v)+ql_A(D(a))v+ql_A(Q(a))v\rangle\\
&=&\langle f, (1+2q)l_A(a)(\alpha+\beta)(v)\rangle.
\end{eqnarray*}}
Therefore, if $q=-\frac{1}{2}$ or Eq. (\mref{eq:cond2}) holds, we
have
$l_{l_A,q}^\ast+r_{l_A,q}^\ast=l_{-l_A^\ast,q}$.
The equality $
r_{-l_A^\ast,q}=-r_{l_A,q}^\ast$ can be
checked similarly.

Note that by Proposition \ref{pro:cond}, $Q+\alpha^\ast$ is a derivation on $(A\ltimes_{-l_A^\ast} V^\ast,\cdot)$ if and only if the following equality
\vspb
\begin{eqnarray}\label{eq:condd}
(a+f)\cdot (D+\beta^\ast)(b+g)=-(a+f)\cdot(Q+\alpha^\ast)(b+g), \;\; a, b\in A,\;\; f, g\in V^\ast,
\end{eqnarray}
holds. Note that
\vspb
\begin{eqnarray*}
(a+f)\cdot (D+\beta^\ast)(b+g)=a\cdot D(b)-l_A^\ast(a)\beta^\ast(g)-l_A^\ast(D(b))f,
\end{eqnarray*}
and 
\vspb
\begin{eqnarray*}
-(a+f)\cdot (Q+\alpha^\ast)(b+g)=-a\cdot Q(b)+l_A^\ast(a)\alpha^\ast(g)+l_A^\ast(Q(b))f.
\end{eqnarray*}
Therefore, Eq. (\ref{eq:condd}) holds  if and only if $a\cdot D(b)=-a\cdot Q(b)$, $ l_{A}^\ast(a)\beta^\ast(g)=-l_A^\ast(a)\alpha^\ast(g)$ and $l_A^\ast(D(b))f=-l_A^\ast(Q(b))f$ for all $a$, $b\in A$ and $f$, $g\in V^\ast$. Since
\vspb
\begin{eqnarray*}
&&\langle l_A^\ast(a)\beta^\ast(g)+l_A^\ast(a)\alpha^\ast(g), v\rangle=-\langle g, (\alpha+\beta)(l_A(a)v))\rangle,
\end{eqnarray*}
 $ l_A^\ast(a)\beta^\ast(g)=-l_A^\ast(a)\alpha^\ast(g)$ holds for all $a\in A$ and $g\in V^\ast$ if and only if $(\alpha+\beta)(l_A(a)v)=0$ for all $a\in A$, $v\in V$. Similarly, $l_A^\ast(D(b))f=-l_A^\ast(Q(b))f$ for all $b\in A$ and $f\in V^\ast$ holds if and only if $l_A((D+Q)(b))v=0$ for all $b\in A$, $v\in V$.
 Therefore, Eq. (\mref{eq:condd}) holds if and only if Eq. (\mref{eq:condd2}) holds. Then the conclusion follows directly.
 \vspb
\end{proof}

\subsection{NYBE via AYBE: pre-Novikov algebras via admissible differential Zinbiel algebras}
We first recall the construction of antisymmetric solutions of the NYBE from pre-Novikov algebras. 
\begin{defi}\mcite{HBG}
Let $A$ be a vector space with binary operations
$\lhd$ and $\rhd$. Then $(A, \lhd, \rhd)$ is called a {\bf pre-Novikov algebra} if
\begin{eqnarray}
&&\mlabel{ND1}a\rhd (b\rhd c)=(a\rhd b+a\lhd b)\rhd c+b\rhd (a\rhd c)-(b\rhd a+b\lhd a)\rhd c,\\
&&\mlabel{ND2} a\rhd (b\lhd c)=(a\rhd b)\lhd c+b\lhd (a\lhd c+a\rhd c)-(b\lhd a)\lhd c,\\
&&\mlabel{ND3} (a\lhd b+a\rhd b)\rhd c=(a\rhd c)\lhd b,\\
&&\mlabel{ND4}(a\lhd b)\lhd c=(a\lhd c)\lhd b, \quad a, b, c\in A.
\end{eqnarray}
\end{defi}

For a pre-Novikov algebra $(A, \lhd, \rhd)$, denote $L_\rhd(a)(b)=a\rhd b$ and $R_\lhd (a)(b)=b\lhd a$ for all $a$, $b\in A$.

\begin{pro} \cite[Propositions 3.33 \& 3.34]{HBG} \mlabel{ndend-o-oper}
\begin{enumerate}
\item Let $(A, \lhd, \rhd)$ be a pre-Novikov algebra. Then a Novikov algebra structure on $A$ is given by
\vspa
\begin{eqnarray}
\mlabel{ND5}\circ: A\otimes A\rightarrow A, \quad
a \circ b\coloneqq a\lhd b+a\rhd b,~~~~\;\;\;a, b\in A.
\end{eqnarray}
$(A, \circ)$ is called {\bf the descendent Novikov algebra} of  $(A, \lhd, \rhd)$.
Moreover, $(A, L_{\rhd}, R_{\lhd})$ is a representation of $(A, \circ)$, and the identity map $\id$ is an $\mathcal O$-operator on $(A, \circ)$ associated to the representation $(A,
L_{\rhd}, R_{\lhd})$.

\item Let $(A, \circ)$ be a Novikov algebra, $(V, l_A,
r_A)$ be a representation of $(A, \circ)$ and $T: V\rightarrow A$ be an
$\mathcal{O}$-operator on $(A, \circ)$ associated to $(V, l_A, r_A)$. Then there
exists a pre-Novikov algebra structure on $V$ defined by
\vspb
\begin{eqnarray}\mlabel{eq:ndend}
u\rhd v=l_A(T(u))v,\;\; u\lhd v= r_A(T(v))u,\;\;\; u,~~v\in V.
\end{eqnarray}
\end{enumerate}
\end{pro}

\begin{pro}\mlabel{NYB-ND}\cite[Theorem 3.35]{HBG}
Let $(A, \lhd, \rhd)$ be a pre-Novikov algebra and $(A,
\circ)$ be the descendent Novikov algebra.  Then for a basis $\{e_1, \ldots, e_n\}$ of $A$ and its dual basis $\{e_1^\ast, \ldots, e_n^\ast\}$ of $A^\ast$,
\vspb
\begin{eqnarray}\mlabel{eq:solu}
r=\sum_{i=1}^n(e_i\otimes e_i^\ast-e_i^\ast\otimes e_i),
\end{eqnarray}
is an antisymmetric solution of  the \nybe
in the Novikov algebra $A\ltimes_{L_\rhd^\ast+R_\lhd^\ast,
-R_\lhd^\ast}A^\ast$.
\end{pro}

Also recall that a vector space $A$ with a binary operation $\diamond$ on $A$ is called a {\bf Zinbiel algebra}~ \mcite{Lod} if
\begin{equation}
a\diamond (b\diamond c)=(b\diamond a)\diamond c+(a\diamond b)\diamond c,\;\;\;\;
a, b, c\in A.
\end{equation}

Furthermore, if $D$ is a derivation on a Zinbiel algebra
$(A, \diamond)$ in the sense that
$$D(a\diamond b)=a\diamond D(b)+D(a)\diamond b,\;\;\;\; a, b\in A,$$
then $(A, \diamond, D)$ is called a {\bf differential Zinbiel algebra}~\mcite{LLB}.

Next, we introduce the notion of admissible differential
Zinbiel algebras.
\begin{defi}
Let $(A, \diamond, D)$ be a differential Zinbiel algebra and $Q:A\rightarrow A$ be a linear map. If
$Q$ satisfies
\vspb
\begin{eqnarray*}
Q(a\diamond b)=Q(a)\diamond b-a\diamond D(b)=a\diamond Q(b)-D(a)\diamond b,\;\;\; a, b\in A,
\end{eqnarray*}
then $(A, \diamond, D, Q)$ is called an {\bf admissible differential Zinbiel algebra}.
\end{defi}

Note that if $(A, \diamond, D)$ is a differential Zinbiel algebra, then $(A, \diamond, D, -D)$ is an admissible differential Zinbiel algebra.

Pre-Novikov algebras can be constructed from admissible differential Zinbiel algebras as follows.
\begin{pro}\mlabel{pre-Nov-constr2}
Let $(A,\diamond, D, Q)$ be an admissible differential Zinbiel algebra and $q\in {\bf k}$.
Define binary operations  $\lhd$ and $\rhd: A\otimes
A\rightarrow A$ by
\begin{equation}\mlabel{constr-ndend-2}
a\lhd_q b:=(D+qQ)(b)\diamond a,\;\;a\rhd_q b:=a\diamond (D+qQ)b,\;\;\;\; a,b\in A.
\end{equation}
Then $(A,\lhd_q,\rhd_q)$ is a pre-Novikov algebra.
\end{pro}
\begin{proof}
Let $a$, $b$, $c\in A$ and $u=(D+qQ)c$. We have
{\wuhao\begin{eqnarray*}
&&a\rhd_q (b\rhd_q c)-(a\rhd_q b+a\lhd_q b)\rhd_q c-b\rhd_q(a\rhd_q c)+(b\rhd_q a+b\lhd_q a)\rhd_q c\\
&&\quad=a\diamond (D+qQ)(b\diamond (D+qQ)c)-(a\diamond (D+qQ)b+(D+qQ)(b)\diamond a)\diamond (D+qQ)c\\
&&\qquad-b\diamond (D+qQ)(a\diamond (D+qQ)c)+(b\diamond (D+qQ)a+(D+qQ)(a)\diamond b)\diamond (D+qQ)c\\
&&\quad=a\diamond (D+qQ)(b\diamond u)-(a\diamond (D+qQ)b+(D+qQ)(b)\diamond a)\diamond u\\
&&\qquad-b\diamond (D+qQ)(a\diamond u)+(b\diamond (D+qQ)a+(D+qQ)(a)\diamond b)\diamond u\\
&&\quad= a\diamond (D(b)\diamond u+b\diamond D(u)+qQ(b)\diamond u-qb\diamond D(u))-(a\diamond D(b)+qa\diamond Q(b)+D(b)\diamond a\\
&&\qquad+qQ(b)\diamond a)\diamond u-b\diamond (D(a)\diamond u+a\diamond D(u)+qQ(a)\diamond u-qa\diamond D(u))\\
&&\qquad+(b\diamond D(a)+qb\diamond Q(a)+D(a)\diamond b+qQ(a)\diamond b)\diamond u\\
&&\quad=(a\diamond (D(b)\diamond u)-(a\diamond D(b)+D(b)\diamond a)\diamond u)+q (a\diamond(Q(b)\diamond u)-(a\diamond Q(b)+Q(b)\diamond a)\diamond u)\\
&&\qquad-(b\diamond (D(a)\diamond u)-(b\diamond D(a)+D(a)\diamond b)\diamond u)-q (b\diamond (Q(a)\diamond u)-(b\diamond Q(a)+Q(a)\diamond b)\diamond u)\\
&&\qquad+(1-q)(a\diamond (b\diamond D(u))-b\diamond (a\diamond D(u)))\\
&&\quad=0.
\end{eqnarray*}}
Therefore, Eq. (\mref{ND1}) holds. Eqs. (\mref{ND2})--(\mref{ND4}) can be similarly checked.
\end{proof}

\begin{rmk}
Note that $(A,\lhd_0,\rhd_0)$ is a pre-Novikov algebra, as
proved in \mcite{HBG}. Moreover, $(A,\lhd_q,\rhd_q)$ can be
regarded as a deformation of $(A,\lhd_0,\rhd_0)$, in the sense of
the ordinary deformation theory.
\end{rmk}

\begin{lem}\mlabel{diff-zinb-algebra-ndend-2}
\begin{enumerate}
\item Let $(A, \diamond, D, Q)$ be an admissible differential Zinbiel algebra. Then there is an admissible commutative differential algebra $(A, \cdot, D, Q)$ with the binary operation defined by
\begin{eqnarray}\mlabel{diff-zinb-algebra-ndend-3}
a\cdot b:=a\diamond b+b\diamond a, \;\;\;\;a, b\in A,
\end{eqnarray}
which is called the {\bf descendent admissible commutative differential algebra of $(A, \diamond, D, Q)$}. Moreover, $(A, L_{ \diamond}, D, Q)$ is a representation of the descendent commutative differential algebra $(A, \cdot, D, Q)$ and the identity map $\id: A\rightarrow A$ is an $\mathcal{O}$-operator on $(A, \cdot, D, Q)$ associated to $(A, L_{\diamond}, D, Q)$.
\item Let $(A, \cdot, D, Q)$ be an admissible commutative differential algebra, $(V, l_A, \alpha, \beta)$ be a representation of $(A, \cdot, D, Q)$ and $T$ be an $\mathcal{O}$-operator on $(A, \cdot, D, Q)$ associated to $(V, l_A, \alpha, \beta)$. Then there exists an admissible differential Zinbiel algebra $(V, \diamond, \alpha, \beta)$ on $V$ with the binary operation defined by
\vspb
\begin{eqnarray}\mlabel{diff-zinb-algebra-ndend-4}
  u\diamond v:=l_A(T(u))v,\;\;\; \;\;u, v\in V.
\end{eqnarray}
\end{enumerate}
\end{lem}
\begin{proof}
It is straightforward.
\end{proof}

\begin{pro}\mlabel{algebra-comm-2} Let $(A, \diamond, D, Q)$
be an admissible differential Zinbiel algebra and $(A, \cdot, D,
Q)$ be the descendent admissible commutative differential algebra
of $(A, \diamond, D, Q)$, with $\cdot$ defined by Eq.
\meqref{diff-zinb-algebra-ndend-3}. Let $(A, \lhd_q, \rhd_q)$ be
the pre-Novikov algebra induced from $(A, \diamond, D, Q)$, where
$\lhd_q, \rhd_q$ are defined by Eq. \meqref{constr-ndend-2}. Then
the descendent Novikov algebra $(A, \circ_q')$ of $(A, \lhd_q,
\rhd_q)$ is the same as the Novikov algebra $(A, \circ_q)$ induced
from $(A, \cdot, D, Q)$.
\end{pro}

\begin{proof}
For $a$, $b\in A$, we have
\vspb
\begin{eqnarray*}
a\circ_q' b=a\lhd_q b+a\rhd_q b=(D+qQ)(b)\diamond a+a\diamond
(D+qQ)(b)=a\cdot (D+qQ)(b)=a\circ_q b.
\end{eqnarray*}
This proves the conclusion.
\end{proof}

\begin{pro}\mlabel{O-oper-comm-2}
Let $(A, \cdot, D, Q)$ be an admissible commutative differential
algebra and $(A, \circ_{q})$ be the Novikov algebra induced from
$(A, \cdot, D, Q)$. Suppose that $T$ is an $\mathcal{O}$-operator
on $(A, \cdot, D, Q)$ associated to a representation $(V, l_A,
\alpha, \beta)$, and $(V, \diamond, \alpha, \beta)$ is the
admissible differential Zinbiel algebra defined by Eq.
\meqref{diff-zinb-algebra-ndend-4}. Then the pre-Novikov algebra
structure on $V$ obtained by the $\mathcal{O}$-operator on $(A,
\circ_{q})$ associated to the representation $(V,
l_{l_A,q}, r_{l_A,q})$ through Eq.
\meqref{eq:ndend} is exactly the one obtained from the
differential Zinbiel algebra $(V, \diamond, \alpha, \beta)$
through Eq. \meqref{constr-ndend-2}.
\end{pro}

\begin{proof}
The pre-Novikov algebra structure on $V$ obtained by the $\mathcal{O}$-operator of $(A, \circ_{q})$ associated to the representation $(V, l_{l_A,q}, r_{l_A,q})$ through Eq. (\mref{eq:ndend}) is given by
\begin{eqnarray*}
&&u \rhd_q v=l_{l_A,q}(T(u))v=l_A(T(u))(\alpha+q\beta)(v)=u\diamond (\alpha+q\beta)(v),\\
&& u\lhd_q v=r_{l_A,q}(T(v))u=l_A((D+qQ)(T(v)))u=l_A(T((\alpha+q\beta)v))u=(\alpha+q\beta)(v)\diamond u, \quad u, v\in V.
\end{eqnarray*}
This yields the conclusion.
\end{proof}

Let $(A, \diamond, D, Q)$ be an admissible differential
Zinbiel algebra and $q$ be in $\bfk$.
Let $(A, \cdot, D, Q)$ be the descendent admissible commutative
differential algebra. On the one hand, by
Proposition~\mref{diff-zinb-algebra-ndend-2}, the identity map
${\rm id}$ is an $\mathcal{O}$-operator on $(A, \cdot, D, Q)$
associated to $(A, L_{\diamond}, D, Q)$.  Let $(A, \circ_{q})$ be
the Novikov algebra induced from $(A, \cdot, D, Q)$ and thus by
Proposition \mref{corr-O-operator-2}, $\id$ is an $\calo$-operator
on $(A, \circ_{q})$ associated to $(A,l_{
L_{\diamond},q}, r_{L_{ \diamond},q})$. On the other
hand, let $(A, \lhd_q, \rhd_q)$ be the pre-Novikov algebra that is
induced from $(A, \diamond, D, Q)$ through  Eq.
\meqref{constr-ndend-2}. By Proposition~\mref{algebra-comm-2}, the
 descendent Novikov algebra of $(A, \lhd_q, \rhd_q)$ is still
$(A,\circ_q)$. Hence by Proposition \mref{ndend-o-oper}, $\id$ is
an $\mathcal{O}$-operator on $(A, \circ_{q})$ associated to $(A,
L_{\rhd_q}, R_{\lhd_q})$.

\begin{pro} Let $(A, \diamond, D, Q)$ be an admissible differential Zinbiel algebra and $q$ be in $\bfk$. With the notations given above, we have
\vspb
$$l_{
L_{\diamond},q}=L_{\rhd_q}, \;\; r_{L_{
\diamond},q}=R_{\lhd_q}.$$ That is, the two approaches to get the
identity map $\id$ as the $\mathcal O$-operators on the Novikov
algebra $(A, \circ_{q})$ coincide, as depicted in the following
commutative diagram.
\vspb
 \begin{displaymath}
\xymatrix{
\txt{\tiny $(A, \diamond, D, Q)$\\ \tiny an admissible differential \\ \tiny Zinbiel algebra } \ar[rr]^-{\rm Lem. \mref{diff-zinb-algebra-ndend-2}}   \ar[d]_-{\rm Prop. \mref{pre-Nov-constr2}}&& \txt{ \tiny $\id$\\ \tiny an $\mathcal{O}$-operator on $(A, \cdot, D, Q)$ \\ \tiny associated to $(A, L_{\diamond}, D, Q)$  } \ar[d]_-{\rm Prop. \mref{corr-O-operator-2}}\\
\txt{\tiny $(A, \lhd_q, \rhd_q)$\\ \tiny a pre-Novikov algebra }
\ar[rr]^-{\rm Prop. \mref{ndend-o-oper}}
            && \txt{\tiny $\id$\\ \tiny an $\mathcal{O}$-operator on $(A, \circ_{q})$ \\ \tiny associated to $(A, l_{L_{ \diamond},q}, r_{L_{\diamond},q})=(A,
L_{\rhd_q}, R_{\lhd_q})$ }}
\end{displaymath}
\mlabel{pp:preoop}
\vspb
\end{pro}

\begin{proof}
Let $a$, $b\in A$. Then we obtain
\begin{eqnarray*}
l_{
L_{\diamond},q}(a)(b)=L_{\diamond}(a)((D+qQ)b)=a\diamond (D+qQ)b=L_{\rhd_q}(a)(b), \\
 r_{L_{
\diamond},q}(a)(b)=L_{\diamond}((D+qQ)a)(b)=(D+qQ)(a)\diamond b=R_{\lhd_q}(a)(b).
\end{eqnarray*}
Hence, $l_{
L_{\diamond},q}=L_{\rhd_q}$, and $ r_{L_{
\diamond},q}=R_{\lhd_q}$. Then the conclusion follows.
\end{proof}

Combining the above results, we obtain the following conclusion that applies admissible differential Zinbiel algebras to construct Novikov bialgebras. To keep track of the  related objects and maps, we use the following diagram,  obtained by merging the diagrams from Propositions~\mref{pp:rmatbialg}, \mref{pro:o-ybe} and \mref{pp:preoop}.
\vspb
\begin{equation}
    \begin{split}
    \xymatrix{
\txt{\tiny $(A, \diamond, D, Q)$\\ \tiny an admissible differential \\ \tiny Zinbiel algebra } \ar[rrr]^-{\txt{\tiny Lems. \mref{diff-zinb-algebra-ndend-2}, \mref{diff-coboundary}}~~~and~~~Prop. \mref{O-Oper-2}}   \ar[d]_-{\txt{\tiny Prop.
\mref{pre-Nov-constr2}}}&&& \txt{ \tiny $(A\ltimes_{-L_{\diamond}^\ast}A^\ast, \cdot, \delta, D+Q^\ast,
Q+D^\ast)$ \\ \tiny \;\; a commutative and
cocommutative \\ \tiny differential ASI bialgebra } \ar[d]_-{\txt{\tiny Prop. \mref{prostr1}}}^-{\txt{\tiny Thm. \mref{bialgebra-constr}}}\\
\txt{\tiny $(A, \lhd_q, \rhd_q)$\\ \tiny a pre-Novikov algebra }
\ar[rrr]^-{\txt{\tiny Props. \mref{NYB-ND}~~~and~~~\mref{Nov-coboundary}}}
&&& \txt{\tiny $(A\ltimes_{L_{\rhd_q}^\ast+R_{\lhd_q}^\ast,
-R_{\lhd_q}^\ast}A^\ast, \circ, \Delta)\stackrel{?}{=}(A\oplus A^\ast, \circ_q, \Delta_q)$ \\ \tiny a Novikov bialgebra  }
            }
    \end{split}
    \mlabel{di:zinb-bialg}
\end{equation}

\begin{thm}\mlabel{Constr-Nov-diff-Zinb-2}
Let $(A, \diamond, D, Q)$ be an admissible differential Zinbiel
algebra and $(A, \cdot, D, Q)$ be the descendent admissible
commutative differential algebra.
\begin{enumerate}
\item Suppose that $(A, \lhd_q,
\rhd_q)$ is the pre-Novikov algebra obtained from $(A, \diamond,
D, Q)$ through Eq. \meqref{constr-ndend-2}, and $(A, \circ_{q})$ is the Novikov algebra induced from $(A, \cdot, D, Q)$. Further let $\{e_1,
\ldots, e_n\}$ be a basis of $A$, $\{e_1^\ast, \ldots, e_n^\ast\}$ be the dual basis of $A^\ast$ and
\begin{eqnarray}
r:=\sum_{i=1}^n(e_i\otimes e_i^\ast-e_i^\ast\otimes e_i).
\end{eqnarray}
Then there is a Novikov bialgebra
$(A\ltimes_{L_{\rhd_q}^\ast+R_{\lhd_q}^\ast,
-R_{\lhd_q}^\ast}A^\ast, \circ, \Delta)$ where $\Delta$ is given
by Eq. \meqref{co1}.
\mlabel{it:zinb-bialg1}
\item There is a commutative and
cocommutative differential ASI bialgebra
$(A\ltimes_{-L_{\diamond}^\ast}A^\ast, \cdot, \delta, D+Q^\ast,
Q+D^\ast)$, where $\delta$ is defined by $-r$ through
Eq.~(\ref{cobass1}).
\mlabel{it:zinb-bialg2}
\item \mlabel{it:zinb-bialg3} If $q=-\frac{1}{2}$ or $Q$ is a derivation on $(A, \diamond)$,
then there is a Novikov bialgebra $(A\oplus A^\ast, \circ_q, \Delta_q)$ obtained from $(A\ltimes_{-L_{\diamond}^\ast}A^\ast, \cdot, \delta, D+Q^\ast,
Q+D^\ast)$ through Eqs. \meqref{eq:k} and
\meqref{eq:cons2}, which coincides with
$(A\ltimes_{L_{\rhd_q}^\ast+R_{\lhd_q}^\ast,
-R_{\lhd_q}^\ast}A^\ast, \circ, \Delta)$.
In other words, Diagram~\meqref{di:zinb-bialg} commutes.
\end{enumerate}
\end{thm}

\begin{proof}
Item (\mref{it:zinb-bialg1}) follows directly by Propositions \mref{pre-Nov-constr2}, \mref{NYB-ND} and \mref{Nov-coboundary}.

Item (\mref{it:zinb-bialg2}) follows directly from Lemmas \mref{diff-zinb-algebra-ndend-2}, \mref{diff-coboundary} and Proposition \mref{O-Oper-2}.

In Item (\mref{it:zinb-bialg3}), for the commutative and
cocommutative differential ASI bialgebra
$(A\ltimes_{-L_{\diamond}^\ast}A^\ast, \cdot, \delta, D+Q^\ast,
Q+D^\ast)$, by the proof of Proposition \ref{pro:o-ybe}, $Q+D^\ast$ is a derivation on $(A\ltimes_{-L_{\diamond}^\ast}A^\ast, \cdot)$ if and only if the following conditions are satisfied.
\begin{eqnarray}\label{eq:conddd}
a\cdot (D+Q)(b)=0,\;\;L_{\diamond}((D+Q)(a))b=0,\;\;(D+Q)(L_{\diamond}(a)b)=0,\;\;\;\;a, b\in A.
\end{eqnarray}
Obviously, Eq. (\mref{eq:conddd}) holds if and only if
\vspb
  \begin{eqnarray}
 \label{cond-Zinb}
a\diamond D(b)=-a\diamond Q(b),\;\;D(a)\diamond b=-Q(a)\diamond
b,\;\;\;a, b\in A.
  \end{eqnarray}
Note that for an admissible differential Zinbiel
algebra $(A, \diamond, D, Q)$, Eq. (\mref{cond-Zinb}) holds if and only if $Q$ is a derivation on $(A, \diamond)$. Then this conclusion follows directly from Proposition \ref{pro:o-ybe} and Theorem \mref{bialgebra-constr}.
\end{proof}

\begin{rmk}
Without giving further details, we remark that combining Theorems \mref{thm-con-Lie-bi} and \mref{Constr-Nov-diff-Zinb-2} gives a natural construction of infinite-dimensional Lie bialgebras from admissible differential Zinbiel algebras.
\end{rmk}
 
We give an example where Theorem~\mref{Constr-Nov-diff-Zinb-2} can be applied to give a deformation family of Novikov bialgebras. Note that a classification of low dimensional Zinbiel algebras over $\CC$ has been achieved in~\mcite{AA}. So more examples can be obtained in the same say.
    \begin{ex} \mlabel{ex:zinb-bialg}
        Consider the Zinbiel algebra $(A=\CC e_1\oplus \CC e_2 \oplus \CC e_3, \diamond)$ from \mcite{AA}, whose non-zero product is given by
\vspb
        \begin{eqnarray*}
            e_1\diamond e_1=e_2, \;\;e_1\diamond e_2=2e_3,\;\; e_2\diamond e_1=e_3.
        \end{eqnarray*}
        Let $D: A\rightarrow A$ be the linear map given by
        \begin{eqnarray*}
            D(e_1)=e_1,\;\;D(e_2)=2e_2,\;\;D(e_3)=3e_3.
        \end{eqnarray*}
        Then $(A, \diamond, D)$ is a differential Zinbiel algebra. Define another linear map $Q: A\rightarrow A$ by
        \begin{eqnarray*}
            Q(e_1)=-e_1+e_3, Q(e_2)=-2e_2,\;\; Q(e_3)=-3e_3.
        \end{eqnarray*}
        Then $(A, \diamond, D, Q)$ is an admissible differential Zinbiel algebra with $Q$ a derivation on $(A,\diamond)$. Then by Theorem~\mref{Constr-Nov-diff-Zinb-2}, the deformation family of Novikov bialgebras $(A\ltimes_{L_{\rhd_{q}}^\ast+R_{\lhd_{q}}^\ast,
            -R_{\lhd_{q}}^\ast}A^\ast, \circ, \Delta)$ agrees with the deformation family of Novikov bialgebras $(A\oplus A^*,\circ_q,\Delta_q)$.
    \end{ex}

We end the paper with a discussion on the construction of Novikov bialgebras from admissible differential Zinbiel algebras that goes beyond Theorem~\mref{Constr-Nov-diff-Zinb-2}. We then illustrate this discussion by an example.

Let $(A,\diamond,D,Q)$ be an admissible differential Zinbiel algebra. Starting from the top-left corner of Diagram~\meqref{di:zinb-bialg} and taking the down-then-right path,  Theorem~\mref{Constr-Nov-diff-Zinb-2} \meqref{it:zinb-bialg1}  gives us a deformation family of Novikov bialgebras $(A\ltimes_{L_{\rhd_{q}}^\ast+R_{\lhd_{q}}^\ast,
    -R_{\lhd_{q}}^\ast}A^\ast, \circ, \Delta)$.
On the other hand, if we follow the right-then-down path of Diagram~\meqref{di:zinb-bialg}, first we have a commutative and cocommutative differential ASI bialgebra, which further gives rise to a triple $(A\oplus A^*,\circ_q,\Delta_q)$ by Eqs. (\mref{eq:k}) and  \meqref{eq:cons2}. Then for a given $q$, there can be three possibilities for the property of $(A\oplus A^*,\circ_q,\Delta_q)$:
\begin{enumerate}
\item $(A\oplus A^*,\circ_q,\Delta_q)$ is a Novikov bialgebra, and it coincides with the Novikov bialgebras $(A\ltimes_{L_{\rhd_{q}}^\ast+R_{\lhd_{q}}^\ast,
    -R_{\lhd_{q}}^\ast}A^\ast, \circ, \Delta)$;
\item $(A\oplus A^*,\circ_q,\Delta_q)$ is a Novikov bialgebra, but it does not coincide with the Novikov bialgebras $(A\ltimes_{L_{\rhd_{q}}^\ast+R_{\lhd_{q}}^\ast,
    -R_{\lhd_{q}}^\ast}A^\ast, \circ, \Delta)$;
\item $(A\oplus A^*,\circ_q,\Delta_q)$ is not a Novikov bialgebra.
\end{enumerate}

In the following example, we see that, for the same admissible commutative Zinbiel algebra, all the three possibilities can indeed occur, depending on the parameter $q\in \bfk$.

\begin{ex}\mlabel{examp-Novikov-2}
Let $(A, \diamond, D)$ be the differential Zinbiel algebra in Example~\mref{ex:zinb-bialg}.
Let $Q:A\rightarrow A$ be the linear map given by
\vspb
            \begin{eqnarray*}
            Q(e_1)=3e_1+e_3,\;\;Q(e_2)=2e_2,\;\;Q(e_3)=e_3.
\end{eqnarray*}
One directly checks that $(A, \diamond, D, Q)$ is an admissible
differential Zinbiel algebra where $Q$ is not a derivation on
$(A, \diamond)$.

We first follow Theorem~\mref{Constr-Nov-diff-Zinb-2} and provide
the steps in constructing a Novikov bialgebras following the down-then-right path in
Diagram~\meqref{di:zinb-bialg}.

Let $(A, \cdot, D, Q)$ be the descendent admissible  commutative differential algebra  of $(A, \diamond, D, Q)$. Its  nonzero product is given by
\vspb
\begin{eqnarray*}
e_1\cdot e_1=2e_2,\;\;e_1\cdot e_2=e_2\cdot e_1=3e_3.
\end{eqnarray*}
Then the nonzero product of the Novikov algebra $(A, \circ_{q})$ induced from $(A, \cdot, D, Q)$ is given by
\begin{eqnarray*}
e_1\circ_{q} e_1=2(1+3q)e_2,\;\;e_1\circ_{q} e_2=6(1+q)e_3,\;\;e_2\circ_{q} e_1=3(1+3q)e_3.
\end{eqnarray*}
Further the pre-Novikov algebra $(A, \lhd_{q}, \rhd_{q})$ induced from $(A, \diamond, D, Q)$ has its nonzero product given by
{\wuhao\begin{eqnarray*}
&&e_1\lhd_{q} e_1=(1+3q)e_2,\;\;e_1 \lhd_{q} e_2=2(1+q)e_3,\;\;e_2\lhd_{q} e_1=2(1+3q)e_3,\\
&& e_1\rhd_{q} e_1=(1+3q)e_2,\;\;e_1\rhd_{q} e_2=4(1+q)e_3,\;\; e_2\rhd_{q} e_1=(1+3q)e_3.
\end{eqnarray*}}
 Taking $\{e_1^\ast, e_2^\ast, e_3^\ast\}$ to be the basis of $A^\ast$ dual to $\{e_1, e_2, e_3\}$, then the nonzero product of the Novikov algebra $(A\ltimes_{L_{\rhd_{q}}^\ast+R_{\lhd_{q}}^\ast,
-R_{\lhd_{q}}^\ast}A^\ast,\circ)$ is given by
{\wuhao\begin{eqnarray*}
\label{eq:exx1}&&e_1\circ e_1=2(1+3q)e_2,\;\;e_1\circ e_2=6(1+q)e_3,\;\;e_2\circ e_1=3(1+3q)e_3,\\
&&e_1\circ e_2^\ast=-2(1+3q)e_1^\ast,\;\;e_2^\ast\circ e_1=(1+3q)e_1^\ast,\;\;e_1\circ e_3^\ast=-2(3+5q)e_2^\ast,\\
\label{eq:exx2}&&e_3^\ast\circ e_1=2(1+3q)e_2^\ast,\;\;e_2\circ e_3^\ast=-(3+5q)e_1^\ast,\;\;e_3^\ast\circ e_2=2(1+q)e_1^\ast.
\end{eqnarray*}}
Set $r=\sum_{i=1}^3(e_i\otimes e_i^\ast-e_i^\ast\otimes e_i)$. Then $r$ is an antisymmetric solution of the NYBE in $(A\ltimes_{L_{\rhd_{q}}^\ast+R_{\lhd_{q}}^\ast,
-R_{\lhd_{q}}^\ast}A^\ast, \circ)$. Then by Proposition \ref{Nov-coboundary}, there is a Novikov bialgebra $(A\ltimes_{L_{\rhd_{q}}^\ast+R_{\lhd_{q}}^\ast,
-R_{\lhd_{q}}^\ast}A^\ast, \circ, \Delta)$, where $\Delta$ is given by
\wuhao{\begin{eqnarray*}
&\Delta(e_1)=(1+3q)e_2\otimes e_1^\ast+2(1+q)e_3\otimes e_2^\ast-2(1+3q)e_1^\ast\otimes e_2-(3+5q)e_2^\ast\otimes e_3,\\
&\Delta(e_2)=2(1+3q)e_3\otimes e_1^\ast-2(3+5q)e_1^\ast\otimes e_3,\;\;\Delta(e_3)=\Delta(e_1^\ast)=0,\\
&\Delta(e_2^\ast)=2(1+3q)e_1^\ast\otimes e_1^\ast,\;\;\Delta(e_3^\ast)=6(1+q)e_1^\ast\otimes e_2^\ast+3(1+3q)e_2^\ast\otimes e_1^\ast.
\end{eqnarray*}}

On the other hand, the admissible differential Zinbiel algebra $(A, \diamond, D, Q)$, following the right arrow in Diagram~\meqref{di:zinb-bialg}, gives the commutative and cocommutative differential ASI bialgebra $(A\ltimes_{-L_{\diamond}^\ast}A^\ast, \cdot,
\delta, D+Q^\ast, Q+D^\ast)$. Its nonzero product and coproduct
are given by
\wuhao{\begin{eqnarray}
\label{eq:ex1}
&e_1\cdot e_1=2e_2,\;\;e_1\cdot e_2=e_2\cdot e_1=3e_3, \;\;e_1\cdot e_2^\ast=e_2^\ast\cdot e_1=e_1^\ast,&\\
&e_1\cdot e_3^\ast=e_3^\ast\cdot e_1=2e_2^\ast,\;\; e_2\cdot e_3^\ast=e_3^\ast\cdot e_2=e_1^\ast,&\\
&\delta(e_1)=e_2\otimes e_1^\ast+e_3\otimes e_2^\ast+e_1^\ast\otimes e_2+e_2^\ast\otimes e_3,\;\;
\delta(e_2)=2e_3\otimes e_1^\ast+2e_1^\ast\otimes e_3,&\\
&\delta(e_3)=\delta(e_1^\ast)=0,\;\;
\label{eq:ex2}\delta(e_2^\ast)=2e_1^\ast\otimes e_1^\ast,\;\;\delta(e_3^\ast)=3e_1^\ast\otimes e_2^\ast+3e_2^\ast\otimes e_1^\ast.&
\end{eqnarray}}
Moreover, we have
\vspb
\begin{eqnarray*}
&&D^\ast(e_1^\ast)=e_1^\ast,\;\;D^\ast(e_2^\ast)=2e_2^\ast,\;\; D^\ast(e_3^\ast)=3e_3^\ast,\;\;Q^\ast(e_1^\ast)=3e_1^\ast,\;\;Q^\ast(e_2^\ast)=2e_2^\ast,\;\;Q^\ast(e_3^\ast)=e_1^\ast+e_3^\ast.
\end{eqnarray*}
Therefore, the nonzero product $\circ_q$ of the Novikov algebra and the coproduct $\Delta_q$ of the Novikov coalgebra induced from  the commutative and cocommutative differential
ASI bialgebra $(A\ltimes_{-L_{\diamond}^\ast}A^\ast, \cdot,
\delta, D+Q^\ast, Q+D^\ast)$ by Eqs. (\mref{eq:k}) and  \meqref{eq:cons2} respectively are given by
\vspb
{\wuhao\begin{eqnarray*}
&e_1\circ_q e_1=2(1+3q)e_2,\;\;e_1\circ_q e_2=6(1+q)e_3,\;\;e_2\circ_q e_1=3(1+3q)e_3,\\
&e_1\circ_q e_2^\ast=2(1+q)e_1^\ast,\;\;e_2^\ast\circ_q e_1=(1+3q)e_1^\ast,\;\;e_1\circ_q e_3^\ast=2(1+3q)e_2^\ast,\\
&e_3^\ast\circ_q e_1=2(1+3q)e_2^\ast,\;\;e_2\circ_q e_3^\ast=(1+3q)e_1^\ast,\;\;e_3^\ast\circ_q e_2=2(1+q)e_1^\ast,\\
&\Delta_q(e_1)=(1+3q)e_2\otimes e_1^\ast+2(1+q)e_3\otimes e_2^\ast+2(1+q)e_1^\ast\otimes e_2+(1+3q)e_2^\ast\otimes e_3,\\
&\Delta_q(e_2)=2(1+3q)e_3\otimes e_1^\ast+2(1+3q)e_1^\ast\otimes e_3,\;\;\Delta_q(e_3)=\Delta_q(e_1^\ast)=0,\\
&\Delta_q(e_2^\ast)=2(1+3q)e_1^\ast\otimes e_1^\ast,\;\;\Delta_q(e_3^\ast)=6(1+q)e_1^\ast\otimes e_2^\ast+3(1+3q)e_2^\ast\otimes e_1^\ast.
\end{eqnarray*}}

Then one can apply Proposition~\mref{prostr1} and directly check that $(A\oplus A^\ast, \circ_q, \Delta_q)$ is a Novikov bialgebra if and only if $q=-\frac{1}{2}$ or $q=-1$.
Furthermore, $\circ=\circ_q$ and $\Delta=\Delta_q$ if and only if $q=-\frac{1}{2}$.

In summary, while the triple $(A\ltimes_{L_{\rhd_{q}}^\ast+R_{\lhd_{q}}^\ast,
-R_{\lhd_{q}}^\ast}A^\ast, \circ, \Delta)$ is a Novikov bialgebra for each $q\in \bfk$; for the triple $(A\oplus A^\ast, \circ_q, \Delta_q)$, there are the following three possibilities:
\begin{enumerate}
\item exactly when $q=-\frac{1}{2}$, the triple is a Novikov bialgebra that agrees with  $(A\ltimes_{L_{\rhd_{q}}^\ast+R_{\lhd_{q}}^\ast,
-R_{\lhd_{q}}^\ast}A^\ast, \circ, \Delta)$. Note that this also
provides the example promised in
    Remark~\mref{rmk:example}\,(\ref{it:a1}); 
\item exactly when $q=-1$, the triple is a Novikov bialgebra that does not agree with
$(A\ltimes_{L_{\rhd_{q}}^\ast+R_{\lhd_{q}}^\ast,
    -R_{\lhd_{q}}^\ast}A^\ast, \circ, \Delta)$.
Note that this also provides the example promised in
Remark~\mref{rmk:example}\,(\ref{it:a2}), since it gives a Novikov bialgebra without the condition that $Q$ is a derivation; 
\item exactly
in the remaining cases, the triple is not a Novikov bialgebra,
providing the examples promised in
    Remark~\mref{rmk:example}\,(\ref{it:a3}).
\end{enumerate}
\end{ex}

\noindent{\bf Acknowledgments.} This research is supported by
NSFC (12171129, 11931009, 12271265, 12261131498, 12326319), and the Fundamental Research Funds for the Central Universities and Nankai Zhide Foundation. We thank Guilai Liu for helpful discussions.

\smallskip

\noindent
{\bf Declaration of interests. } The authors have no conflicts of interest to disclose.

\smallskip

\noindent
{\bf Data availability. } Data sharing is not applicable to this article as no new data were created or analyzed in this study.

\end{document}